\documentclass{article}

\usepackage{amsmath,amssymb}
\usepackage{amsthm}
\usepackage[pdftex,bookmarks=true]{hyperref}
\usepackage{cite}

\usepackage{tikz}
\tikzset{
    every node/.style={draw, circle, inner sep=1pt}, 
    every label/.style={rectangle, draw=none, blue}, 
    selected/.style={fill=red!40}, 
    edge weight/.style={midway, rectangle, fill=white, draw=none}
}
\usetikzlibrary{arrows}

\usepackage{soul}
\usepackage{cancel}
\usepackage{enumitem}

\newtheorem{theorem}{Theorem}[section]
\newtheorem{lemma}[theorem]{Lemma}
\newtheorem{proposition}[theorem]{Proposition}
\newtheorem{corollary}[theorem]{Corollary}

\theoremstyle{definition}
\newtheorem{definition}[theorem]{Definition}
\newtheorem{observation}[theorem]{Observation}
\newtheorem{remark}[theorem]{Remark}
\newtheorem{example}[theorem]{Example}

\newtheorem{problem}[theorem]{Problem}
\newtheorem{algorithm}[theorem]{Algorithm}

\newcommand{\trans}{^\top}

\newcommand{\bzero}{\mathbf{0}}
\newcommand{\bone}{\mathbf{1}}

\newcommand{\bb}{\mathbf{b}}

\newcommand{\be}{\mathbf{e}}

\newcommand{\bx}{\mathbf{x}}
\newcommand{\by}{\mathbf{y}}
\newcommand{\bz}{\mathbf{z}}

\newcommand{\bw}{\mathbf{w}}

\newcommand{\Col}{\operatorname{Col}}

\newcommand{\vspan}{\operatorname{span}}

\newcommand{\mptn}{\mathcal{S}}
\newcommand{\diag}{\operatorname{diag}}
\newcommand{\wtt}{\mathcal{T}_2}
\newcommand{\wtot}{\mathcal{T}_1^{(2)}}
\newcommand{\mul}{\operatorname{mult}}

\title{Inverse Fiedler vector problem of a graph}
\author{
Jephian C.-H.~Lin
\thanks{Department of Applied Mathematics, National Sun Yat-sen University, Kaohsiung 80424, Taiwan (jephianlin@gmail.com)}
\and 
Mahsa N Shirazi
\thanks{Department of Mathematics, University of Manitoba, Winnipeg, MB, Canada, R3T 2N2 (mahsa.nasrollahi@gmail.com)}}
\date{\today}

\begin{document}

\maketitle

\begin{abstract}
Given a graph and one of its weighted Laplacian matrix, a Fiedler vector is an eigenvector with respect to the second smallest eigenvalue.  The Fiedler vectors have been used widely for graph partitioning, graph drawing, spectral clustering, and finding the characteristic set.  This paper studies how the graph structure can control the possible Fiedler vectors for different weighted Laplacian matrices.  For a given tree, we characterize all possible Fiedler vectors among its weighted Laplacian matrix.  As an application, the characteristic set can be anywhere on a tree, except for the set containing a single leaf.  For a given cycle, we characterize all possible eigenvectors corresponding to the second or the third smallest eigenvalue.   
\end{abstract}  


\medskip


\section{Introduction}
\label{sec:intro}
Let $G$ be a graph on $n$ vertices.  The set $\mptn(G)$ collects all $n\times n$ real symmetric matrices whose off-diagonal $i,j$-entry is nonzero if and only if $\{i,j\}$ is an edge of $G$.  Note that there are no restrictions on the diagonal entries.  The \emph{inverse eigenvalue problem of a graph $G$} (IEP-$G$) studies the possible spectra among matrices in $\mptn(G)$, aiming to explore the relations between the graph and its associated spectral properties.  The research on the IEP-$G$ garners lots of attentions and has produced fruitful results; see, e.g., the monograph \cite{IEPGZF22} and the references there in. 

Aside from its theoretical interests, another motivation of the IEP-$G$ is from the vibration theory; see, e.g., \cite{GladwellIPiV05}.   In 1974, Hochstadt~\cite{Hochstadt74} studied the vibration of a string and called  the matrices in $\mptn(P_n)$ with nonnegative off-diagonal entries as the Jacobi matrices, where $P_n$ is the path on $n$ vertices.  In 1976, Gray and Wilson~\cite{GW76} and Hald~\cite{Hald76} independently showed that a multiset of $n$ real numbers is the spectrum of some Jacobi matrices if and only if all numbers are distinct.  In 1980, Ferguson~\cite{Ferguson80} studied the vibration of a ring and called the matrices in $\mptn(C_n)$ with nonnegative off-diagonal matrices as the periodic Jacobi matrices; the same paper characterized all possible spectra of a periodic Jacobi matrix.  In general, any spring-mass system can be modeled by a simple graph where the vertices and the edges represent the masses and the springs, respectively.  Given a spring-mass system represented by a graph $G$, its vibration behavior is governed by the differential equation $M\ddot{\bx} = -L \bx$, where $M$ is a diagonal matrix determined by the weights of the masses, $\bx$ is a vector recording the displacements of the masses, and $L$ is a matrix in $\mptn(G)$ with nonpositive off-diagonal entries; see, e.g., \cite{GladwellIPiV05} for more details.  While the IEP-$G$ shifted its attention to $\mptn(G)$, these matrices in $\mptn(G)$ with a fixed sign on the off-diagonal nonzero entries hold stronger connections to the vibration theory and remain an interesting set to be explored.

This paper studies the special case when every mass has unit weight and when the system is not tied with other system (free ends).  Thus, $M = I$ and $L$ has row sum zero on each row.  That is, $L$ is in the subset of $\mptn(G)$ defined by   
\[
    \mptn_L(G) = \{A \in \mptn(G): A\text{ has nonpositive off-diagonal entries and }A\bone = \bzero\},
\]
where $\bone$ is the all-ones vector and $\bzero$ is the zero vector.  In this case, the differential equation becomes $\ddot{\bx} = -L\bx$.  Solving the differential equation, the eigenvalues and the eigenvectors of $L$ describe the frequencies and the modes of the vibration, respectively.  

Matrices in $\mptn_L(G)$ are known as the weighted Laplacian matrices of $G$, whose definition and properties will be reviewed later.  For a given weighted Laplacian matrix, an eigenvector with respect to the second smallest eigenvalue is called a Fiedler vector \cite{Fiedler75class}, and it has been used widely for graph partitioning \cite{PSL90, CCSz97, UZ14}, graph drawing \cite{Hall70, Koren05}, spectral clustering \cite{SM00, vLuxburg07}, and many real-world applications.  It is also used to define the characteristic set of a tree, which can be viewed as the central part of the tree \cite{Merris87}.

This paper studies the \emph{inverse Fiedler vector problem of a graph}:  Given a graph $G$, what are the possible Fiedler vectors for matrices in $\mptn_L(G)$?  With the various applications of the Fiedler vector, researching this problem also give answers to the possible partitions, the possible drawings, the possible characteristic sets of a given graph.

The paper is organized as follows.  In Section~\ref{sec:tree}, we characterize all possible Fiedler vectors of a given tree.  As an application, we show that the characteristic set can be anywhere except for a singleton on a leaf.  We also introduce a bijective transformation in Section~\ref{sec:connsub} between the Type II trees and the Type I trees whose characteristic set is a singleton on a degree-$2$ vertex, which simplifies the analysis of the characteristic sets from two types into one type.  In Section~\ref{sec:cycle}, we characterize all possible eigenvectors of a cycle corresponding to the second or the third smallest eigenvalue.  The rest of this section will be devoted to some terminologies and background.

\subsection{Preliminaries}
\label{ssec:prelim}
Let $X$ and $Y$ be finite sets.  We write $\mathbb{R}^X$ for the set of vectors whose entries are indexed by $X$ and $\mathbb{R}^{X\times Y}$ for the set of matrices whose rows and columns are indexed by $X$ and $Y$, respectively.  In particular, any vector in $\mathbb{R}^X$ can be viewed as a function from $X$ to $\mathbb{R}$.  Given a real symmetric matrix $A$, we write $\lambda_k(A)$ for its $k$-th smallest eigenvalue and $\mul_A(\lambda)$ for the multiplicity of $\lambda$ as an eigenvalue of $A$.  

Let $G$ be a simple graph.  A \emph{weight assignment} or \emph{weight vector} on $G$ is a vector $\bw\in\mathbb{R}^{E(G)}$ that is entrywisely positive; that is, it assigns a positive real number to each edge.  A \emph{weighted graph} is a pair $(G,\bw)$ of a simple graph along with some weight assignment.  For convenience, we write the weight of the edge $\{i,j\}$ as   
\[
    \bw(i,j) = \bw(j,i)
\]
instead of $\bw(\{i,j\})$.  The \emph{weighted Laplacian matrix} of the weighted graph $(G,\bw)$ is the symmetric matrix in $\mathbb{R}^{V(G)\times V(G)}$ whose off-diagonal $i,j$-entry is $-\bw(i,j)$ if $\{i,j\}\in E(G)$ and $0$ if $\{i,j\}\notin E(G)$ and whose diagonal $i,i$-entry is $\displaystyle\sum_{k: \{i,k\}\in E(G)}\bw(i,k)$.  Any of such matrix is called a weighted Laplacian matrix of $G$.  

By definition, any weighted Laplacian matrix $A$ of $G$ satisfies $A\bone = \bzero$.  Therefore, $\mptn_L(G)$ is the collection of all weighted Laplacian matrix of $G$.  When each edge receives weight $1$, we say it is an unweighted graph, and the corresponding Laplacian matrix is the classical Laplacian matrix.

The weighted Laplacian matrices of a graph has many well-known properties; see, e.g., \cite{BapatGM14, BLS07}.  Let $G$ be a graph and $A\in\mptn_L(G)$.  It is known that $A$ is positive semidefinite and $\lambda_1(A) = 0$.  Also, the nullity of $A$ is exactly equal to the number of components of $G$.  As a result, $\lambda_2(A) > 0$ if and only if $G$ is connected, so $\lambda_2$ is known as the \emph{algebraic connectivity} \cite{Fiedler73}.  A \emph{Fiedler vector} is an eigenvector of $A$ with respect to $\lambda_2(A)$ \cite{Fiedler75class}, where Fiedler called it the \emph{characteristic valuation}.  Note that both the algebraic connectivity and the Fiedler vector depend on the weight assignment, while we say $\bx$ is a Fiedler vector of $G$ if it is a Fiedler vector of some $A\in\mptn_L(G)$.  Fiedler showed \cite{Fiedler75class} that if $G$ is connected and $\bx = \begin{bmatrix} x_i \end{bmatrix}\in\mathbb{R}^{V(G)}$ is a Fiedler vector of $G$, then for any $c \leq 0$ the induced subgraph of $G$ on $V_{\geq c}$ is connected, where $V_{\geq c} = \{i\in V(G): x_i \geq c\}$.  Therefore, $\{V_{\geq c}, V(G)\setminus V_{\geq c}\}$ gives a reasonable partition of $G$, which leads to many different applications \cite{PSL90, CCSz97, UZ14, SM00}. 

Motivated by the fruitful applications of the Fiedler vector, it is natural to ask how graph structure shapes the potential Fiedler vector among all possible weight assignments.  

\begin{problem}[Inverse Fiedler vector problem of a graph]
\label{prob:ifpg}
Given a graph $G$, what are the possible Fiedler vectors among matrices in $\mptn_L(G)$?
\end{problem}

Various places in the literature show some hints that Fiedler had considered this problem in mind.  
In \cite[Theorem~3.8]{Fiedler75class}, the potential partitions $\{V_{\geq c}, V\setminus V_{\geq c}\}$ are considered.  The Remark in the end of \cite{Fiedler75class} mentioned the potential Fiedler vectors, but no proof was given.  
The goal of this paper is to provide rigorous proofs and a more comprehensive study on this topic.

When $(T,\bw)$ is a weighted tree, its Fiedler vector $\bx = \begin{bmatrix} x_i \end{bmatrix}\in\mathbb{R}^{V(T)}$ has even richer structure.  It is known that exactly one of the following two cases will occur \cite{Fiedler75ac, KNS96}: 
\begin{description}
\item[Type I] Some entries of $\bx$ are zero.  In this case, there is exactly a vertex $i$ with $x_i = 0$ that is adjacent some vertex $j$ with $x_j \neq 0$.  Moreover, for any path in $T$ starting at $i$, the values of $\bx$ along the path is either strictly increasing, strictly decreasing, or constantly zero.  
We say $\{i\}$ is the \emph{characteristic set} of $(T,\bw)$.
\item[Type II] No entry of $\bx$ is zero. In this case, there is exactly an edge $\{i,j\}$ with $x_ix_j < 0$, say $x_i < 0 < x_j$.  Moreover, for any path starting at $i$ without passing $j$, the values of $\bx$ along the path is strictly decreasing; for any path starting at $j$ without passing $i$, the values of $\bx$ along the path is strictly increasing.  
We say $\{i,j\}$ is the \emph{characteristic set} of $(T,\bw)$.
\end{description}
Note that when the algebraic connectivity has multiplicity bigger than one, there are infinitely many choices of $\bx$, but the theory showed that the characteristic set is independent of the choice of the Fiedler vectors \cite{Fiedler75ac, KNS96}.

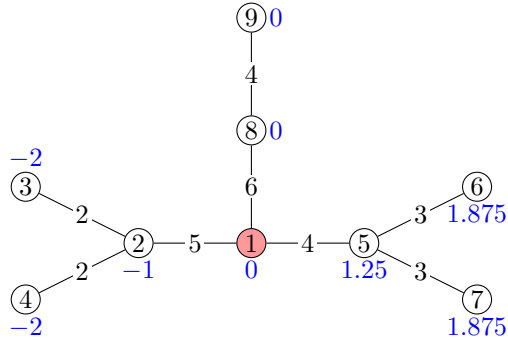
\begin{figure}[h]
\centering
\begin{tikzpicture}[scale=1.5]
\node[selected, label={below:$0$}] (1) at (0, 0) {$1$};
\node[label={below:$-1$}] (2) at (-1, 0) {$2$};
\node[label={above:$-2$}] (3) at (-2, 0.5) {$3$};
\node[label={below:$-2$}] (4) at (-2, -0.5) {$4$};
\node[label={below:$1.25$}] (5) at (1, 0) {$5$};
\node[label={below:$1.875$}] (6) at (2, 0.5) {$6$};
\node[label={below:$1.875$}] (7) at (2, -0.5) {$7$};
\node[label={right:$0$}] (8) at (0, 1) {$8$};
\node[label={right:$0$}] (9) at (0, 2) {$9$};
\draw (3) --node[edge weight]{$2$} (2) --node[edge weight]{$2$} (4);
\draw (6) --node[edge weight]{$3$} (5) --node[edge weight]{$3$} (7);
\draw (2) --node[edge weight]{$5$} (1) --node[edge weight]{$4$} (5);
\draw (1) --node[edge weight]{$6$} (8) --node[edge weight]{$4$} (9);
\end{tikzpicture}
\caption{A weighted tree with a Type I Fiedler vector (marked in blue) and its characteristic set (marked in red).}
\label{fig:t1treechar}
\end{figure}

\begin{figure}[h]
\centering
\begin{tikzpicture}[scale=1.5]
\node[selected, label={below:$-1$}] (1) at (-0.5, 0) {$1$};
\node[label={above:$-2$}] (2) at (-1.5, 0.5) {$2$};
\node[label={below:$-2$}] (3) at (-1.5, -0.5) {$3$};
\node[selected, label={below:$1.25$}] (4) at (0.5, 0) {$4$};
\node[label={above:$1.875$}] (5) at (1.5, 0.5) {$5$};
\node[label={below:$1.875$}] (6) at (1.5, -0.5) {$6$};
\draw (2) --node[edge weight]{$2$} (1) --node[edge weight]{$2$} (3);
\draw (5) --node[edge weight]{$3$} (4) --node[edge weight]{$3$} (6);
\draw (1) --node[edge weight]{$\frac{20}{9}$} (4);
\end{tikzpicture}
\caption{A weighted tree with a Type II Fiedler vector (marked in blue) and its characteristic set (marked in red).}
\label{fig:t2treechar}
\end{figure}
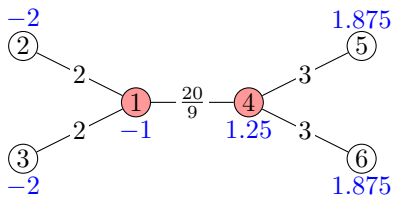

\begin{example}
\label{ex:treechar}
Let $T_1$ and $T_2$ be the trees in Figure~\ref{fig:t1treechar} and \ref{fig:t2treechar}, respectively.  The numbers on the edges are the weight assignments.  Then  
\[
    L_1 = \begin{bmatrix}
    15 & -5 & 0 & 0 & -4 & 0 & 0 & -6 & 0 \\
    -5 & 9 & -2 & -2 & 0 & 0 & 0 & 0 & 0 \\
    0 & -2 & 2 & 0 & 0 & 0 & 0 & 0 & 0 \\
    0 & -2 & 0 & 2 & 0 & 0 & 0 & 0 & 0 \\
    -4 & 0 & 0 & 0 & 10 & -3 & -3 & 0 & 0 \\
    0 & 0 & 0 & 0 & -3 & 3 & 0 & 0 & 0 \\
    0 & 0 & 0 & 0 & -3 & 0 & 3 & 0 & 0 \\
    -6 & 0 & 0 & 0 & 0 & 0 & 0 & 10 & -4 \\
    0 & 0 & 0 & 0 & 0 & 0 & 0 & -4 & 4
    \end{bmatrix}
\]
and 
\[
    L_2 = \begin{bmatrix}
    \frac{56}{9} & -2 & -2 & -\frac{20}{9} & 0 & 0 \\
    -2 & 2 & 0 & 0 & 0 & 0 \\
    -2 & 0 & 2 & 0 & 0 & 0 \\
    -\frac{20}{9} & 0 & 0 & \frac{74}{9} & -3 & -3 \\
    0 & 0 & 0 & -3 & 3 & 0 \\
    0 & 0 & 0 & -3 & 0 & 3
    \end{bmatrix}
\]
are the corresponding weighted Laplacian matrices.  By direct computation, the algebraic connectivity is $1$ with multiplicity $1$ for each of the weighted trees, and the corresponding Fiedler vectors are the blue values by the vertices in Figures~\ref{fig:t1treechar} and \ref{fig:t2treechar}.  Figure~\ref{fig:t1treechar} is an example of the Type I Fiedler vector.  In this example, the characteristic set is $\{1\}$, and for any path starting at $1$, the values on the Fiedler vector is either strictly increasing, strictly decreasing, or constantly zero.  Figure~\ref{fig:t2treechar} is an example of the Type II Fiedler vector.  Here the characteristic set is $\{1,4\}$, and for any path starting at $1$ without passing $4$, the values are strictly decreasing, for any path starting at $4$ without passing $1$, the values are strictly increasing.
\end{example}

Let $G$ be a graph.  An \emph{orientation} of $G$ is a directed graph obtained from $G$ by replacing each undirected edge $\{i,j\}$ with a directed edge, either $(i,j)$ or $(j,i)$.  The \emph{incidence matrix} of an orientation of $G$ is the matrix in $\mathbb{R}^{V(G)\times E(G)}$ such that the $e$-th column has a unique $1$ at the $i$-th entry, a unique $-1$ at the $j$-th entry, and otherwise $0$ if the edge $e = \{i,j\}$ is oriented as $(i,j)$.  An incidence matrix of $G$ is the incidence matrix of some orientation of $G$.  It is known \cite{AM85} that the classical Laplacian matrix of $G$ can be written as $NN\trans$, where $N$ can be any incidence matrix of $G$.  It is straightforward to see that the weighted Laplacian matrix of a weighted graph $(G,\bw)$ can be written as $NWN\trans$, where $W = \diag(\bw)$ is called the \emph{weight diagonal matrix}.

Let $A$ be a matrix in $\mathbb{R}^{X\times Y}$, $\alpha\subseteq X$, and $\beta\subseteq Y$.  The submatrix of $A$ induced on the rows in $\alpha$ and columns in $\beta$ is denoted by $A[\alpha,\beta]$.  The submatrix of $A$ obtained by removing the rows in $\alpha$ and columns in $\beta$ is denoted by $A(\alpha,\beta)$.  When $\alpha = \beta$, we simply write $A[\alpha]$ and $A(\beta)$.  When $\alpha = \beta = \{i\}$, we write $A[i]$ and $A(i)$ to make the notation easier.  Subvectors are defined in a similar way.  Finally, in this paper we use the notation $\bx \oslash \by$ for the entrywise division of two vectors of the same length.  That is, when $\bx = \begin{bmatrix} x_i \end{bmatrix}$ and $\by = \begin{bmatrix} y_i \end{bmatrix}$ are vectors in $\mathbb{R}^X$ with $\by$ nowhere zero, we have $\bx \oslash \by = \begin{bmatrix} \frac{x_i}{y_i} \end{bmatrix}\in\mathbb{R}^X$.  

Throughout the paper, we will use the Perron--Frobenius theorem and the Cauchy interlacing theorem frequently.  The reader may refer to \cite{BHSoG12,BapatGM14} for these standard results.

\section{The inverse Fiedler vector problem of a tree}
\label{sec:tree}

In this section, we provide the complete solution to Problem~\ref{prob:ifpg} when the graph is a tree.  As we have seen in Subsection~\ref{ssec:prelim}, a Fiedler vector of a tree $T$ necessarily have some sign conditions on its entries and the monotonicity following the tree structure.  We will show that these conditions are enough to characterize all Fiedler vectors among matrices in $\mptn_L(T)$.

\begin{definition}
Let $T$ be a tree on $n$ vertices.  A vector $\bx = \begin{bmatrix} x_i \end{bmatrix}\in\mathbb{R}^{V(T)}$ is said to be \emph{Fiedler-like} with respect to $T$ if $\bone\trans\bx = 0$ and one of the following two conditions holds:
\begin{description}
\item[Type I] There is exactly a vertex $i$ with $x_i = 0$ that is adjacent to some vertex $j$ with $x_j \neq 0$.  And for any path in $T$ starting at $i$, the values of $\bx$ along the path is either strictly increasing, strictly decreasing, or constantly zero.  
We say $\{i\}$ is the \emph{characteristic set} of $\bx$.  
\item[Type II] There is exactly an edge $\{i,j\}$ with $x_ix_j < 0$, say $x_i < 0 < x_j$.  And for any path starting at $i$ without passing $j$, the values of $\bx$ along the path is strictly decreasing; for any path starting at $j$ without passing $i$, the values of $\bx$ along the path is strictly increasing.  
We say $\{i,j\}$ is the \emph{characteristic set} of $(T,\bw)$.
\end{description}
\end{definition}

\begin{theorem}
\label{thm:treemain}
Let $T$ be a tree.  Then $\bx$ is a Fiedler vector of $T$ if and only if $\bx$ is Fiedler-like with respect to $T$.
\end{theorem}

By \cite{Fiedler75ac, KNS96}, every Fiedler vector of a tree has to be Fiedler-like.  Conversely, we will show that for every Fiedler-like vector $\bx$ there is a matrix $A\in\mptn_L(T)$ such that $\bx$ is its Fiedler vector in Theorems~\ref{thm:typeif} and \ref{thm:typeiif}, which take care of Type I and Type II Fiedler-like vectors, respectively.  Thus, Theorem~\ref{thm:treemain} will be immediate by then.

Let $T$ be a tree, $i\in V(T)$, and $A\in\mptn_L(T)$.  The inverse of $A(i)$ is called the \emph{bottleneck matrix} at $i$.  The bottleneck matrix is known to be entrywisely positive and is proved to be a powerful tool for finding the location of the characteristic set \cite{KNS96}.  In the following, we will take a detour to focus on eigenvector of $A(i)$ with respect to the smallest eigenvalue.  Afterward our algorithm for solving the inverse Fiedler vector problem will rely on these results.

\subsection{Dirichlet matrices and its Perron vectors}
Let $G$ be a graph and $\partial V$ a subset of $V$.  A matrix of the form $A(\partial V)$ with $A\in\mptn_L(G)$ is called a \emph{Dirichlet matrix} of $G$ with the boundary $\partial V$, which have been used to model the vibration with given boundary condition on $\partial V$ \cite{BLS07}.  

Here we focus on the Dirichlet matrix of a tree with its boundary on a leaf.  Let $T$ be a tree, $r$ a leaf of $T$, and $A\in\mptn_L(T)$ the weighted Laplacian matrix of $(T,\bw)$.  Given $\bx = \begin{bmatrix} x_i \end{bmatrix}\in\mathbb{R}^{V(T - r)}$, recall that the quadratic form of the Dirichlet matrix is  
\[
    \bx\trans A(r) \bx = \bw(r,u)x_u^2 + \sum_{\substack{\{i,j\}\in E(T) \\ \{i,j\}\neq \{r,u\}}} \bw(i,j)(x_i - x_j)^2 \geq 0,
\]
where $u$ is the unique neighbor of $r$.  This quantity is never zero unless $x_u = 0$ and every entry of $\bx$ is the same, which means $\bx = \bzero$.  Therefore, $A(r)$ is a positive definite matrix.  

On the other hand, observe that $cI - A(r)$ is a nonnegative irreducible matrix when $c$ is large enough, so the smallest eigenvalue of $A(r)$ is simple by the Perron--Frobenius theorem.  Therefore, we call the smallest eigenvalue of $A(r)$ as the \emph{Perron value} and call its corresponding eigenvector as \emph{Perron vector}.  

By the Perron--Frobenius theorem, the Perron vector can be chosen to be entrywisely positive.  We will see that the values of the Perron vector is strictly increasing along any path starting at $u$.  Moreover, any such vector $\bx$ leads to a unique Dirichlet matrix whose Perron vector is $\bx$.  

\begin{figure}[h]
\centering
\begin{tikzpicture}[scale=1.5]
\node[selected] (1) at (0, 0) {$1$};
\node (2) at (1, 0) {$2$};
\node (3) at (2, 0.5) {$3$};
\node (4) at (2, -0.5) {$4$};
\draw (3) -- (2) -- (4);
\draw (2) -- (1);
\end{tikzpicture}
\caption{A tree with its boundary on a leaf $r = 1$}
\label{fig:treer}
\end{figure}
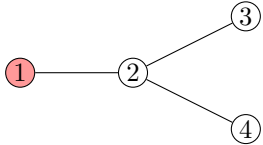

\begin{example}
\label{ex:dirchwtop}
Let $T$ be the graph as in Figure~\ref{fig:treer} with the boundary $r = 1$.  With the weight assignment  
\[
    \bw(1,2) = 5 \text{ and } \bw(2,3) = \bw(3,4) = 2,
\]
the Dirichlet matrix of $T$ with the boundary $r$ is  
\[
    A(r) = \begin{bmatrix}
    9 & -2 & -2 \\
    -2 & 2 & 0 \\
    -2 & 0 & 2
    \end{bmatrix}.
\]
By direct computation, $A(r)$ is positive definite with the smallest eigenvalues $\lambda = 1$ and the Perron vector $\bx = (x_2,x_3,x_4)\trans = (1,2,2)\trans$.  For any path in $T - r$ starting at the unique neighbor $2$ of $r$, the values of $\bx$ along the path are strictly increasing.  
\end{example}

\begin{example}
\label{ex:dirchptow}
Let $T$ be the tree in Figure~\ref{fig:treer} with the boundary $r = 1$, where the weight on each edge is undetermined.  Given a vector $\bx = (x_2,x_3,x_4)\trans = (1,2,2)\trans$, we would like to find a Dirichlet matrix $A(r)$ with $A\in\mptn_L(T)$ such that $\bx$ is its Perron vector.  Note that we may assume the Perron value of $A(r)$ is $\lambda = 1$ since we may replace $A(r)$ with $\frac{1}{\lambda}A(r)$.

By ordering the edges as $e_1 = \{1,2\}$, $e_2 = \{2,3\}$, $e_3 = \{2,4\}$, we may assume $\bw = (w_1, w_2, w_3)\trans$ and $W = \diag(\bw)$ are the weight vector and the weight diagonal matrix, respectively.  Let $N$ be the incidence matrix of $T$ with respect to the orientation where every edge is pointing to $r$.  Then any $A\in\mptn_L(G)$ can be written as $A = NWN\trans$ and $A(r)$ can be written as  
\[
    A(r) = N_rWN_r\trans = 
    \begin{bmatrix}
    1 & -1 & -1 \\
    0 & 1 & 0 \\
    0 & 0 & 1
    \end{bmatrix}
    \begin{bmatrix}
    w_1 & 0 & 0 \\
    0 & w_2 & 0 \\
    0 & 0 & w_3
    \end{bmatrix}
    \begin{bmatrix}
    1 & 0 & 0 \\
    -1 & 1 & 0 \\
    -1 & 0 & 1
    \end{bmatrix},
\]
where $N_r$ is obtained from $N$ by removing the $r$-th row.  Observe that $N_r$ is invertible.  With the assumption $\lambda = 1$, the equation $A(r)\bx = \lambda \bx$ becomes $N_rWN_r\trans \bx = \bx$ and $W(N_r\trans\bx) = (N_r^{-1}\bx)$.  By direct computation, 
\[
    N_r\trans\bx = 
    \begin{bmatrix} 1 \\ 1 \\ 1 \end{bmatrix}
    \text{ and }
    N_r^{-1}\bx = 
    \begin{bmatrix} 5 \\ 2 \\ 2 \end{bmatrix}.
\]
Thus, there is a unique solution $w_1 = 5$, $w_2 = 2$, and $w_3 = 2$.  Therefore, with $\lambda = 1$, the Dirichlet matrix in Example~\ref{ex:dirchwtop} is the only matrix having $\bx$ as the Perron vector.  
\end{example}

In the following, we describe the general strategy in solving the Dirichlet matrix from a Perron vector.  We also provide the combinatorial interpretations of $N_r\trans\bx$ and $N_r^{-1}\bx$.  

\begin{definition}
Let $T$ be a tree with its boundary on a leaf $r$, where $u$ is the unique neighbor of $r$.  Let $N$ be the incidence matrix of $T$ with respect to the orientation where every edge is pointing to $r$.  The \emph{Dirichlet incidence matrix} $N_r$ is obtained from $N$ by removing its $r$-th row.  An entrywisely positive vector $\bx\in\mathbb{R}^{V(T - r)}$ is said to be \emph{strictly increasing} if for any path in $T - r$ starting at $u$, the values of $\bx$ is strictly increasing.
\end{definition}

\begin{proposition}
\label{prop:nrtx}
Let $T$ be a tree with its boundary on a leaf $r$, where $u$ is the unique neighbor of $r$.  Let $N_r$ be the Dirichlet incidence matrix and $\bx = \begin{bmatrix} x_i \end{bmatrix}\in\mathbb{R}^{V(T - r)}$.  Then the $e$-th entry of $N_r\bx$ is
\[
    (N_r\trans\bx)_e = \begin{cases}
        x_u & \text{ if } e = \{r,u\} \text{ and }, \\
        x_j - x_i & \text{ if } e = \{i,j\} \text{ and $i$ is on the path between $j$ and $r$}.
    \end{cases}
\]
\end{proposition}
\begin{proof}
When $e = \{r,u\}$, the $e$-th row of $N_r\trans$ has a unique $1$ at the $u$-th entry and is $0$ otherwise.  When $e = \{i,j\}$ with $i$ on the path between $j$ and $r$, the $e$-th row of $N_r\trans$ has a unique $1$ at the $j$-th entry and a unique $-1$ at the $i$-th entry and is $0$ otherwise.  Then the desired formula follows from direct computation.  
\end{proof}

The notion of the path matrix can be found in, e.g., \cite[Section~2.3]{BapatGM14}.  Note that the original definition of a path matrix does not require $r$ being a leaf.  Here we only define the path matrix on a specific case where $r$ is a leaf and all edges are pointing to $r$.

\begin{definition}
Let $T$ be a tree with its boundary on a leaf $r$.  The \emph{path matrix} of $T$ with respect to $r$ is a matrix in $\mathbb{R}^{E(T)\times V(T - r)}$ such that the $e,i$-entry is $1$ if $e$ is on the path from $i$ to $r$ and $0$ otherwise.
\end{definition}

\begin{theorem}
{\rm \cite[Theorem~2.10]{BapatGM14}}
\label{thm:pathmatrix}
Let $T$ be a tree with its boundary on a leaf $r$.  Let $N_r$ be the Dirichlet incidence matrix and $P$ the path matrix.  Then $N_r^{-1} = P$.  
\end{theorem}

\begin{example}
Let $T$ be the tree in Figure~\ref{fig:treer} with the boundary $r = 1$.  Then  
\[
    N_r = \begin{bmatrix}
    1 & -1 & -1 \\
    0 & 1 & 0 \\
    0 & 0 & 1
    \end{bmatrix}
    \text{ and }
    P = \begin{bmatrix}
    1 & 1 & 1 \\
    0 & 1 & 0 \\
    0 & 0 & 1
    \end{bmatrix}, 
\]
where the rows of $P$ are indexed by $e_1 = \{1,2\}$, $e_2 = \{2,3\}$, $e_3 = \{2,4\}$ and the columns are indexed by $2$, $3$, $4$.  By direct computation, we have $N_rP = I$.  
\end{example}

\begin{corollary}
\label{cor:ninvx}
Let $T$ be a tree with its boundary on a leaf $r$, where $u$ is the unique neighbor of $r$.  Let $N_r$ be the Dirichlet incidence matrix and $\bx = \begin{bmatrix} x_i \end{bmatrix}\in \mathbb{R}^{V(T - r)}$.  Then the $e$-th entry of $N_r^{-1}\bx$ is  
\[
    (N_r^{-1}\bx)_e = \sum_i x_i,
\]
where the sum is over all vertices on the component of $T - e$ not containing $r$.
\end{corollary}
\begin{proof}
By Theorem~\ref{thm:pathmatrix}, $N_r^{-1}\bx = P\bx$, where $P$ is the path matrix of $T$ with respect to $r$.  For each $e$, the graph $T - e$ contains exactly two components, one containing $r$ while the other does not.  The $e$-th row of $P$ has $1$'s on the entries corresponding to vertices in the component of $T - e$ not containing $r$.  By direct computation, the desired formula follows.  
\end{proof}

With these formulas, we are able to construct the Dirichlet matrix with a given Perron vector.  

\begin{theorem}
\label{thm:dirichletbi}
Let $T$ be a tree with its boundary on a leaf $r$, where $N_r$ is its Dirichlet incidence matrix.  A vector $\bx\in\mathbb{R}^{V(T - r)}$ is the Perron vector of some Dirichlet matrix $A(r)$ with $A\in\mptn_L(T)$ if and only if $\bx$ is strictly increasing on $T$ with the boundary $r$.  Moreover, given any strictly increasing vector $\bx$ and a positive $\lambda$, the matrix $N_rWN_r\trans$ is the unique Dirichlet matrix with Perron value $\lambda$ and Perron vector $\bx$, where $W = \diag(\bw)$ and $\bw = \lambda(N_r^{-1}\bx) \oslash (N_r\trans\bx)$.
\end{theorem}
\begin{proof}
Let $A(r)$ be a Dirichlet matrix with $A\in\mptn_L(T)$.  Let $\bx$ be a Perron vector of $A(r)$, where we chose $\bx$ to be entrywisely positive.  We first show that $\bx$ is strictly increasing.  Note that $A(r) = N_rWN_r\trans$ for some weight diagonal matrix $W$.  Let $\lambda > 0$ be the Perron value.  Then $N_rWN_r\trans\bx = \lambda\bx$, which implies $WN_r\trans\bx = \lambda N_r^{-1}\bx$.  Since $\bx$ is entrywisely positive, $N_r^{-1}\bx$ is entrywisely positive as well by Corollary~\ref{cor:ninvx}.  Since $W$ is a diagonal matrix with all diagonal entries positive, we have $N_r\trans\bx$ is entrywisely positive, which implies $\bx$ is strictly increasing by Proposition~\ref{prop:nrtx}.

Conversely, given a strictly increasing vector $\bx$, we may solve the equation $N_rWN_r\trans\bx = \lambda\bx$ for $W$ as follows.  Since $\bx$ is strictly increasing, $N_r\trans\bx$ is entrywisely positive by Proposition~\ref{prop:nrtx}.  Since $\bx$ is entrywisely positive and $\lambda$ is positive, $\lambda N_r^{-1}\bx$ is entrywisely positive by Corollary~\ref{cor:ninvx}.  Since $W = \diag(\bw)$ is a diagonal matrix, the only solution is $\bw = \lambda(N_r^{-1}\bx) \oslash (N_r\trans\bx)$, and the previous discussion guarantees that $\bw$ is entrywisely positive, which makes it a valid weight vector.  Moreover, an entrywisely positive eigenvector necessarily corresponds to the smallest eigenvalue of $A(r)$ by the Perron--Frobenius theorem.  
\end{proof}

Thus, the set of Dirichlet matrices and the set of pairs $(\lambda, \vspan(\{\bx\}))$ with $\lambda > 0$ and $\bx$ strictly increasing are in bijective relation.

\begin{corollary}
\label{cor:dirunique}
Let $T$ be a tree with its boundary on a leaf $r$.  Every Dirichlet matrix $A(r)$ with $A\in\mptn_L(T)$ uniquely determines the eigenpair $(\lambda, \vspan(\{\bx\}))$ of the Perron value and the corresponding one-dimensional eigenspace generated by a strictly increasing Perron vector $\bx$.  Conversely, a pair $(\lambda, \vspan(\{\bx\}))$ of a value $\lambda > 0$ and a one-dimensional subspace generated by a strictly increasing vector $\bx$ uniquely determines the Dirichlet matrix $A(r)$ with $A\in\mptn_L(T)$ having $\lambda$ as the Perron value and $\vspan(\{\bx\})$ as the corresponding eigenspace.
\end{corollary}

\subsection{Weighted Laplacian matrices and Fiedler vectors}

Using Theorem~\ref{thm:dirichletbi}, we may construct the weighted Laplacian matrices from the Fiedler vectors.  Let $T$ be a tree, $v\in V(T)$, and $T_1, \ldots, T_k$ the components of $T - v$ with $k = \deg_T(v)$.  A \emph{branch} of $T$ at $v$ is the subgraph of $T$ induced on $V(T_i)\cup\{v\}$ for some $i$ with the boundary $v$.  

\begin{algorithm}[Type I Fiedler-like vector]
\label{alg:typeif}
Let $T$ be a tree and $\bx\in\mathbb{R}^{V(T)}$ a Type I Fiedler-like vector.  A weighted Laplacian matrix $A\in\mptn_L(G)$ with $A\bx = \bx$ can be constructed by the following steps.
\begin{enumerate}
\item Let $\{r\}$ be the characteristic set of $\bx$.  Let $B_1, \ldots, B_k$ be the branches of $T$ at $r$.  Let $V_i = V(B_i) \setminus \{r\}$ for $i = 1,\ldots, k$.  
\item For $i = 1,\ldots, k$, 
\begin{enumerate}
    \item if the subvector $\bx[V_i]$ is entrywisely positive, solve the weight assignment $\bw_i$ for $B_i$ such that its Dirichlet matrix $A_i$ has Perron value $\mu_i = 1$ and Perron vector $\bx_i = \bx[V_i]$, 
    \item if the subvector $\bx[V_i]$ is entrywisely negative, solve the weight assignment $\bw_i$ for $B_i$ such that its Dirichlet matrix $A_i$ has Perron value $\mu_i = 1$ and Perron vector $\bx_i = -\bx[V_i]$, and
    \item if the subvector $\bx[V_i]$ is constantly zero, choose an arbitrary strictly increasing vector $\bx_i$ and solve the weight assignment $\bw_i$ for $B_i$ such that its Dirichlet matrix $A_i$ has Perron value $\mu_i \geq 1$ and Perron vector $\bx_i$.
\end{enumerate}
\item Since each edge of $T$ only appears in exactly one branch, $\bw_1, \ldots, \bw_k$ naturally define a weight assignment $\bw$ of $T$.  Return the weighted Laplacian matrix of $(T,\bw)$.   
\end{enumerate}
\end{algorithm}

\begin{theorem}
\label{thm:typeif}
Given a tree $T$ and a Type I Fiedler-like vector $\bx\in\mathbb{R}^{V(T)}$, Algorithm~\ref{alg:typeif} always generates a matrix $A\in\mptn_L(T)$ such that $A\bx = \bx$ with $1 = \lambda_2(A)$.  
\end{theorem}
\begin{proof}
We first examine that every step of Algorithm~\ref{alg:typeif} works.  By the definition of a Type I Fiedler-like vector of $T$, the characteristic set of $\bx$ contains a unique vertex $r$.  Moreover, the subvector $\bx[V_i]$ is either entrywisely positive and strictly increasing, entrywisely negative and strictly decreasing (meaning $-\bx[V_i]$ is strictly increasing), or constantly zero.  By Theorem~\ref{thm:dirichletbi}, there is a weight assignment for each branch such that the corresponding Dirichlet matrix $A_i$ has Perron value $\mu_i \geq 1$ and Perron vector $\bx_i$.  By assembling the weight assignments of the branches into a weight assignment of $T$, the algorithm always give a matrix $A\in\mptn_L(T)$.  

By the construction and assuming $r = 1$, the output matrix $A$ and the vector $\bx$ can be conformally written as  
\[
    A = \begin{bmatrix}
        ? & \bb_1\trans & \cdots & \bb_k\trans \\
        \bb_1 & A_1 & ~ & ~ \\
        \cdots & ~ & \ddots & ~ \\
        \bb_k & ~ & ~ & A_k
    \end{bmatrix}
    \text{ and }
    \bx = \begin{bmatrix}
    0 \\ \bx[V_1] \\ \vdots \\ \bx[V_k]
    \end{bmatrix},
\]
where $\bb_i = -A_i\bone$.  Since $A_i\bx_i = \bx_i$ if $\bx[V_i]$ is nonzero, $A_i\bx[V_i] = \bx[V_i]$.  For those $i$ with $\bx[V_i] = \bzero$, $A_i\bx[V_i] = \bzero = \bx[V_i]$ holds as well.  Therefore, $A\bx = \bx$.

Finally, we show that $1 = \lambda_2(A)$.  We have seen that $A(r)$ is the direct sum of matrices $A_1, \ldots, A_k$, whose eigenvalues are $\mu_i$'s.  By the choices of $\mu_i$'s, all eigenvalues of $A(r)$ are greater than or equal to $1$.  By the Cauchy interlacing theorem, there is at most one eigenvalue of $A$ that is strictly smaller than $1$, which is $0$ since $A\in\mptn_L(G)$ and $A\bone = \bzero$.  Thus, $1 = \lambda_2(A)$.  
\end{proof}

\begin{figure}[h]
\centering
\begin{tikzpicture}[scale=1.5]
\begin{scope}[xshift=-1cm]
\node[selected] (1l) at (0, 0) {$1$};
\node[label={below:$1$}] (2) at (-1, 0) {$2$};
\node[label={above:$2$}] (3) at (-2, 0.5) {$3$};
\node[label={below:$2$}] (4) at (-2, -0.5) {$4$};
\end{scope}
\draw (3) --node[edge weight]{$\frac{2}{1}$} (2) --node[edge weight]{$\frac{2}{1}$} (4);
\draw (2) --node[edge weight]{$\frac{5}{1}$} (1l);

\begin{scope}[xshift=1cm]
\node[selected] (1r) at (0, 0) {$1$};
\node[label={below:$1.25$}] (5) at (1, 0) {$5$};
\node[label={below:$1.875$}] (6) at (2, 0.5) {$6$};
\node[label={below:$1.875$}] (7) at (2, -0.5) {$7$};
\end{scope}
\draw (6) --node[edge weight]{$\frac{1.875}{0.625}$} (5) --node[edge weight]{$\frac{1.875}{0.625}$} (7);
\draw (1r) --node[edge weight]{$\frac{5}{1.25}$} (5);

\node[selected] (1m) at (0, 0) {$1$};
\node[label={right:$1$}] (8) at (0, 1) {$8$};
\node[label={right:$2$}] (9) at (0, 2) {$9$};

\draw (1m) --node[edge weight]{$\frac{3}{1}$} (8) --node[edge weight]{$\frac{2}{1}$} (9);
\end{tikzpicture}
\caption{Three branches and their $\bx_i$ (marked in blue) for Example~\ref{ex:t1treerecover}, where on each edge the denominator records of $N_r\trans\bx_i$ and the numerator records $N_r^{-1}\bx_i$.}
\label{fig:t1treerecover}
\end{figure}
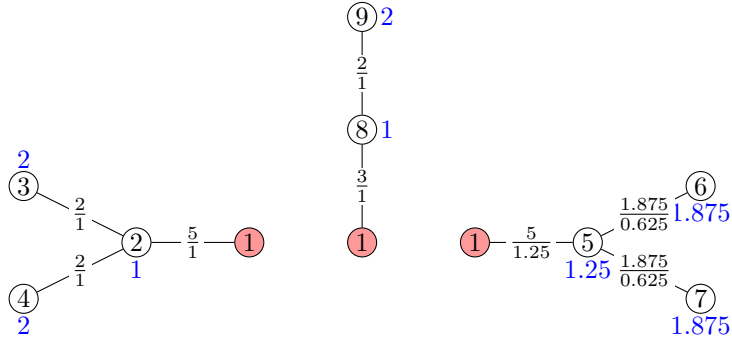

\begin{example}
\label{ex:t1treerecover}
Let $T$ be the tree in Figure~\ref{fig:t1treechar} and  
\[
    \bx = (0,-1,-2,-2,1.25,1.875,1.875,0,0)\trans \in \mathbb{R}^{V(T)}
\]
Let $B_1$, $B_2$, and $B_3$ be the branches of $T$ at $r = 1$ containing vertices $2$, $5$, and $8$, respectively.  By Algorithm~\ref{alg:typeif}, $\bx_1 = (1,2,2)\trans\in \mathbb{R}^{\{2,3,4\}}$ and $\bx_2 = (1.25, 1.875, 1.875)\trans\in\mathbb{R}^{\{5,6,7\}}$, while $\bx_3$ can be chosen to be any strictly increasing vector, e.g., $\bx_3 = (1,2)\trans\in\mathbb{R}^{\{8,9\}}$.  For each of $\bx_i$, we may solve a weight assignment $\bw_i$ using the formula $(N_r^{-1}\bx_i) \oslash (N_r\trans\bx_i)$ by Theorem~\ref{thm:dirichletbi}; these numbers give the fractions on each edge in Figure~\ref{fig:t1treerecover}.  These weights define a weighted tree based on $T$, which has $1$ as the algebraic connectivity and $\bx$ as the Fiedler vector.  Moreover, we still have the freedom to multiply the weights on $B_3$ by a ratio $\mu_3 \geq 1$.  By choosing $\mu_3 = 2$, we obtain the weight assignment as in Figure~\ref{fig:t1treechar}.
\end{example}

Now we consider the Type II Fiedler-like vectors.  We first make an observation, which works for general weighted graphs.  

\begin{lemma}
\label{lem:charcontract}
Let $(G,\bw)$ be a weighted graph with a vertex $r$ of degree $2$, where $p$ and $q$ are its neighbors.  Let $(G',\bw')$ be the weighted graph obtained from $G$ by removing $r$ and adding the edge $\{p,q\}$ with the weight  
\[
    \bw'(p,q) = \frac{\bw(p,r)\bw(q,r)}{\bw(p,r) + \bw(q,r)}.
\]
Let $A$ and $A'$ be the weighted Laplacian matrices of $G$ and $G'$, respectively.  If $A\bx = \lambda\bx$ for some $\bx\in\mathbb{R}^{V(G)}$ and the $r$-th entry of $\bx$ is $0$, then by removing the $r$-th entry from $\bx$, the resulting vector $\bx'\in\mathbb{R}^{V(G')}$ has $A'\bx' = \lambda\bx'$.  
\end{lemma}
\begin{proof}
Let $x_i$ be the $i$-th entry of $\bx$.  For convenience, we also let $w_p = \bw(p,r)$ and $w_q = \bw(q,r)$.  Examining the $r$-th row of $A\bx = \lambda\bx$, we have $-w_px_p - w_qx_q = \lambda x_r = 0$ and, equivalently,
\begin{equation}
\label{eq:wx}
    w_px_p + w_qx_q = 0.
\end{equation}

Let $s = w_p + w_q$.  Define an elementary matrix $E\in\mathbb{R}^{V(T)\times V(T)}$ such that the diagonal entries are $1$'s, the $p,r$-entry is $\frac{w_p}{s}$, the $q,r$-entry is $\frac{w_q}{s}$, and otherwise $0$.  By direct computation and Equation~\ref{eq:wx}, we have $E\bx = \bx$ and $E\trans\bx = \bx$.   

On the other hand, we may check $EAE\trans = \begin{bmatrix} s \end{bmatrix} \oplus A'$, where $\begin{bmatrix} s \end{bmatrix} \oplus A'$ is obtained from $A'$ by padding a new $r$-th row/column, which is zero except that the $r,r$-entry is $s$.  To see this, we focus on the principal submatrices of $EAE\trans$ induced on $\{r,p,q\}$ and observe that 
\[
    \begin{aligned}
    &\mathrel{\phantom{=}}
    \begin{bmatrix}
        1 & 0 & 0 \\
        \frac{w_p}{s} & 1 & 0 \\
        \frac{w_q}{s} & 0 & 1 
    \end{bmatrix}
    \begin{bmatrix}
        s & -w_p & -w_q \\
        -w_p & w_p + \square & 0 \\
        -w_q & 0 & w_q + \triangle
    \end{bmatrix}
    \begin{bmatrix}
        1 & \frac{w_p}{s} & \frac{w_q}{s} \\
        0 & 1 & 0 \\
        0 & 0 & 1 
    \end{bmatrix} \\
    &=     
    \begin{bmatrix}
        s & 0 & 0 \\
        0 & \frac{w_pw_q}{s} + \square & -\frac{w_pw_q}{s} \\
        0 & -\frac{w_pw_q}{s} & \frac{w_pw_q}{s} + \triangle 
    \end{bmatrix},
    \end{aligned}
\]
where $w_p + \square$ and $w_q + \triangle$ indicate the $p,p$-entry and the $q,q$-entry of $A$, respectively.  Note that $\bw'(p,q) = \frac{w_pw_q}{s}$.  Then it is straightforward to check that $EAE\trans = \begin{bmatrix} s \end{bmatrix} \oplus A'$.

Finally, $A\bx = \lambda\bx$, $E\bx = \bx$, $E\trans\bx = \bx$ imply that $(\begin{bmatrix} s \end{bmatrix} \oplus A')\bx = \lambda\bx$.  By removing the $r$-th entry from $\bx$, the resulting vector $\bx'$ has $A'\bx' = \lambda\bx'$.  
\end{proof}

We complement Lemma~\ref{lem:charcontract} with Remark~\ref{rem:elec} to provide some intuition behind the formula.  

\begin{figure}[h]
\centering
\begin{tikzpicture}
\node (0) at (-2,0) {};
\node (1) at (0,0) {};
\node (2) at (4,0) {};
\draw (0) -- (1) -- (2);
\begin{scope}[yshift=0.5cm, every node/.style={rectangle,draw=none}]    
\node at (-1,0) {$R_1$};
\node at (2,0) {$R_2$};
\node[right] at (4,0) {total resistance $R = R_1 + R_2$};
\end{scope}
\begin{scope}[yshift=-0.5cm, every node/.style={rectangle,draw=none}]    
\node at (-1,0) {$w_1 = \frac{1}{R_1}$};
\node at (2,0) {$w_2 = \frac{1}{R_2}$};
\node[right] at (4,0) {total weight $w = \frac{w_1w_2}{w_1 + w_2}$};
\end{scope}
\end{tikzpicture}
\caption{Two resistors in series.}
\label{fig:series}
\end{figure}
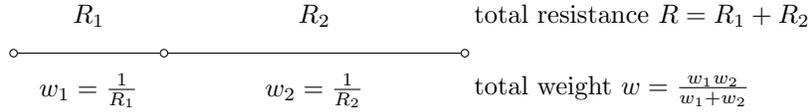

\begin{remark}
\label{rem:elec}
The weighted Laplacian matrices are often used to model an electronic circuit, where the weights are the reciprocals of the resistances; see, e.g., \cite{BF13,GBS08,Devriendt22}.  As illustrated in Figure~\ref{fig:series}, the total resistance $R$ of two resistors in series, with resistances $R_1$ and $R_2$, respectively, is $R = R_1 + R_2$.  With $w = \frac{1}{R}$, $w_1 = \frac{1}{R_1}$, and $w_2 = \frac{1}{R_2}$, intuitively, the total weight of two edges in series is 
\[
    w = \frac{1}{\frac{1}{w_1} + \frac{1}{w_2}} = \frac{w_1w_2}{w_1 + w_2},
\]
justifying the formula in Lemma~\ref{lem:charcontract}.
\end{remark}

Algorithm~\ref{alg:typeiif} can be viewed as a collaboration of Algorithm~\ref{alg:typeif} and Lemma~\ref{lem:charcontract}.  Here we need a specific part of the Perron--Frobenius theorem to compare the spectral radii $\rho$.  

\begin{theorem}
{\rm \cite[Theorem~2.2.1]{BHSoG12}}
\label{thm:pf}
Let $A$ and $B$ be nonnegative matrices such that $B$ is irreducible.  Then the following holds.
\begin{itemize}
\item If $B - A$ is entrywisely nonnegative and $A \neq B$, then $\rho(A) < \rho(B)$. 
\item If $A$ is a proper principal submatrix of $B$, then $\rho(A) < \rho(B)$.  
\end{itemize}
\end{theorem}

\begin{algorithm}[Type II Fiedler-like vector]
\label{alg:typeiif}
Let $T$ be a tree and $\bx = \begin{bmatrix} x_i \end{bmatrix}\in\mathbb{R}^{V(T)}$ a Type II Fiedler-like vector.  A weighted Laplacian matrix $A\in\mptn_L(G)$ with $A\bx = \bx$ can be constructed by the following steps.
\begin{enumerate}
\item Let $\{p,q\}$ be the characteristic set of $\bx$, say $x_p < 0 < x_q$.  Let $B_1$ be the branch at $q$ containing $p$, and $B_2$ the branch at $p$ containing $q$. 
 Let $V_1 = V(B_1)\setminus \{q\}$ and $V_2 = V(B_2)\setminus \{p\}$.  
\item For $i = 1,2$, 
\begin{enumerate}
    \item if the subvector $\bx[V_i]$ is entrywisely positive, solve the weight assignment $\bw_i$ for $B_i$ such that its Dirichlet matrix $A_i$ has Perron value $1$ and Perron vector $\bx_i = \bx[V_i]$, and 
    \item if the subvector $\bx[V_i]$ is entrywisely negative, solve the weight assignment $\bw_i$ for $B_i$ such that its Dirichlet matrix $A_i$ has Perron value $1$ and Perron vector $\bx_i = -\bx[V_i]$.
\end{enumerate}
\item Define the weight assignment $\bw$ for $T$ such that $\bw(e) = \bw_i(e)$ if $e\in E(B_i)$ and $e \neq \{p,q\}$ for $i = 1, 2$, and  
\[
    \bw(p,q) = \frac{\bw_1(p,q)\bw_2(p,q)}{\bw_1(p,q) + \bw_2(p,q)}.
\]
Return the weighted Laplacian matrix of $(T,\bw)$.   
\end{enumerate}
\end{algorithm}

\begin{theorem}
\label{thm:typeiif}
Given a tree $T$ and a Type II Fiedler-like vector $\bx\in\mathbb{R}^{V(T)}$, Algorithm~\ref{alg:typeiif} always generates a matrix $A\in\mptn_L(T)$ such that $A\bx = \bx$ and $1 = \lambda_2(A)$.  
\end{theorem}
\begin{proof}
We first examine that every step of Algorithm~\ref{alg:typeiif} works.  By the definition of a Type II Fiedler-like vector of $T$, we may assume the characteristic set is $\{p,q\}$.  Moreover, the subvector $\bx[V_i]$ is either entrywisely positive and strictly increasing, or entrywisely negative and strictly decreasing.  By Theorem~\ref{thm:dirichletbi}, there is a weight assignment for each branch such that the corresponding Dirichlet matrix $A_i$ has Perron value $1$ and Perron vector $\bx_i$.  Following the definition of Step 3, the algorithm always gives a matrix $A\in\mptn_L(T)$.  

Now we check $A\bx = \bx$ with the help of Algorithm~\ref{alg:typeif} and Lemma~\ref{lem:charcontract}.  Let $T^{(0)}$ be the tree obtained from $T$ by subdividing the edge $\{p,q\}$ into two edges $\{p,r\}$ and $\{q,r\}$, where a new vertex $r$ is added.  Let $\bx^{(0)}$ be the vector obtained from $\bx$ by padding the new $r$-th entry as $0$.  Then Algorithm~\ref{alg:typeif} gives a weight assignment $\bw^{(0)}$ for $T^{(0)}$ such that its weighted Laplacian matrix $A^{(0)}$ has $A^{(0)}\bx^{(0)} = \bx^{(0)}$.  Note applying Algorithm~\ref{alg:typeif} to $T^{(0)}, \bx^{(0)}$ and applying Algorithm~\ref{alg:typeiif} to $T,\bx$ give the same intermediate weight assignments $\bw_1$ and $\bw_2$.  The only difference is that $\bw^{(0)}$ simply combines $\bw_1$ and $\bw_2$, while $\bw$ merge $\bw_1$ and $\bw_2$ by defining a new $\bw(p,q)$.  Therefore, by Lemma~\ref{lem:charcontract}, $A\bx = \bx$.  

Finally, we show that $1$ is indeed the second smallest eigenvalue of $A$.  Let $C_i = A[V_i]$ for $i = 1,2$.  We observe that $\bw(p,q) < \bw_1(p,q)$ by definition and the fact $\frac{\bw_2(p,q)}{\bw_1(p,q) + \bw_2(p,q)} < 1$; similarly, $\bw(p,q) < \bw_2(p,q)$.  Therefore, $C_i$ is obtained from $A_i$ by subtracting a positive value from the a diagonal entry (the $p,p$-entry for $i = 1$ and the $q,q$-entry for $i = 2$).  By the Perron--Frobenius Theorem (Theorem~\ref{thm:pf}), $\lambda_1(C_i) < \lambda_1(A_i) = 1$.  Again by the Perron--Frobenius Theorem (Theorem~\ref{thm:pf}), since $C_1(p)$ is a principal submatrix of $A_1$, $\lambda_1(C_1(p)) > \lambda_1(A_1) = 1$; similarly $\lambda_1(C_2(q)) > \lambda_1(A_2) = 1$.  By the Cauchy interlacing theorem, $\lambda_2(C_1) \geq \lambda_1(C_1(p)) > 1$ and $\lambda_2(C_2) \geq \lambda_1(C_2(q)) > 1$.  Now observe that $A(p)$ is the direct sum of $C_1(p)$ and $C_2$.  Thus, the smallest eigenvalue of $A(p)$ is the smallest eigenvalue of $C_2$, which is smaller than $1$.  The second smallest eigenvalue of $A(r)$ is either the smallest eigenvalue of $C_1(p)$ or the second smallest eigenvalue of $C_2$, while both of them are greater than $1$.  By the Cauchy interlacing theorem, $1$ is the second smallest eigenvalue of $A$.   
\end{proof}

\begin{figure}[h]
\centering
\begin{tikzpicture}[scale=1.5]
\begin{scope}[xshift=-1cm]
\node[label={below:$1$}] (1l) at (-0.5, 0) {$1$};
\node[selected] (4l) at (0.5, 0) {$4$};
\node[label={above:$2$}] (2) at (-1.5, 0.5) {$2$};
\node[label={below:$2$}] (3) at (-1.5, -0.5) {$3$};
\end{scope}
\draw (2) --node[edge weight]{$\frac{2}{1}$} (1l) --node[edge weight]{$\frac{2}{1}$} (3);
\draw (1l) --node[edge weight](14l){$\frac{5}{1}$} (4l);

\begin{scope}[xshift=1cm]
\node[selected] (1r) at (-0.5, 0) {$1$};
\node[label={below:$1.25$}] (4r) at (0.5, 0) {$4$};
\node[label={above:$1.875$}] (5) at (1.5, 0.5) {$5$};
\node[label={below:$1.875$}] (6) at (1.5, -0.5) {$6$};
\end{scope}
\draw (5) --node[edge weight]{$\frac{1.875}{0.625}$} (4r) --node[edge weight]{$\frac{1.875}{0.625}$} (6);
\draw (1r) --node[edge weight](14r){$\frac{5}{1.25}$} (4r);

\node[selected] at (-1,1) (1u) {$1$};
\node[selected] at (1,1) (4u) {$4$};
\draw (1u) --node[edge weight](14u){$\frac{20}{9} = \frac{5\cdot 4}{5 + 4}$} (4u);
\draw[-triangle 45, bend right, shorten >=5pt] (14l.north east) to (14u.south);
\draw[-triangle 45, bend left, shorten >=5pt] (14r.north west) to (14u.south);
\end{tikzpicture}
\caption{Two branches and their $\bx_i$ (marked in blue) for Example~\ref{ex:t2treerecover}, where on each edge the denominator records $N_p\trans\bx_1$ (or $N_q\trans\bx_2$) and the numerator records $N_p^{-1}\bx_1$ (or $N_q^{-1}\bx_2$).}
\label{fig:t2treerecover}
\end{figure}
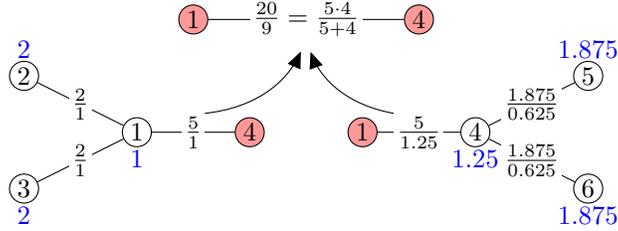

\begin{example}
\label{ex:t2treerecover}
Let $T$ be the tree in Figure~\ref{fig:t2treechar} and  
\[
    \bx = (-1,-2,-2,1.25,1.875,1.875)\trans \in \mathbb{R}^{V(T)}
\]
Let $B_1$ be the branch of $T$ at $q = 4$ containing $p = 1$, and $B_2$ the branch at $1$ containing $4$.  By Algorithm~\ref{alg:typeiif}, $\bx_1 = (1,2,2)\trans\in \mathbb{R}^{\{1,2,3\}}$ and $\bx_2 = (1.25, 1.875, 1.875)\trans\in\mathbb{R}^{\{4,5,6\}}$.  We may solve the weight assignments $\bw_1 = (N_p^{-1}\bx_1) \oslash (N_p\trans\bx_1)$ and $\bw_2 = (N_q^{-1}\bx_2) \oslash (N_q\trans\bx_2)$ by Theorem~\ref{thm:dirichletbi}; these numbers give the fractions on each edge in Figure~\ref{fig:t2treerecover}.  Note that $\bw_1(p,q) = \frac{5}{1} = 5$ and $\bw_2(p,q) = \frac{5}{1.25} = 4$ are usually different.  We define the weight assignment $\bw$ of $T$ by $\bw(p,q) = \frac{5\cdot 4}{5 + 4} = \frac{20}{9}$ while $\bw(e) = \bw_i(e)$ if $e\in E(B_i)$ and $e \neq \{p,q\}$ for $i = 1,2$.  Thus, the weighted Laplacian matrix of $(T,\bw)$ has $\bx$ as its Fiedler vector with respect to the algebraic connectivity $1$.  
\end{example}

Theorems~\ref{thm:typeif} and \ref{thm:typeiif} allow us to put the characteristic set on almost everywhere.

\begin{corollary}
Let $T$ be a tree and $R$ be a set $\{r\}$ such that $r$ is not a leaf or a set $\{p,q\}$ such that $\{p,q\}\in E(T)$.  Then there is a weight assignment $\bw$ such that the characteristic set of the weighted tree $(T,\bw)$ is on $R$.  
\end{corollary}
\begin{proof}
Observe that a Type I Fiedler-like vector allows its characteristic set $\{r\}$ anywhere as long as $r$ is not a leaf.  Similarly, a Type II Fiedler-like vector allows it characteristic set $\{p,q\}$ anywhere as long as $\{p,q\}\in E(T)$.  Then the statement is immediate from Theorems~\ref{thm:typeif} and \ref{thm:typeiif}.
\end{proof}

\subsection{Complete solutions}

In this subsection, we show that Algorithms~\ref{alg:typeif} and \ref{alg:typeiif} generate all possible solutions.  

Let $G$ be a graph and $A\in\mptn(G)$.  A vertex $v\in V(G)$ is said to a \emph{Parter vertex} with respect to $\lambda$ if $\mul_{A(v)}(\lambda) = \mul_A(\lambda) + 1$; see, e.g., \cite{JDSPaper, ParterPaper, WienerPaper}.  Though not stated explicitly, many places in literature, e.g., \cite{Fiedler75ac, KNS96, BHSoG12}, have observed that the unique vertex in the characteristic set of a Type I Fiedler vector is always a Parter vertex and used the fact in the proofs.  Here we state the general form and include the proof for completeness.

\begin{lemma}
\label{lem:charparter}
Let $T$ be a tree, $A\in\mptn(T)$, and $A\bx = \lambda\bx$.  If $\bx$ is zero at $r\in V(T)$ and nonzero at some neighbor $p$ of $r$, then $r$ is a Parter vertex with respect to $\lambda$.    
\end{lemma}
\begin{proof}
Let $C$ be the component of $T - r$ containing $p$.  By the structure of $T$ and $A\bx = \lambda \bx$, we have $A[C]\bx[C] = \lambda\bx[C]$.  In other words, $\bx[C]$ is in $\ker(A[C] - \lambda I)$ and $\bx[C]$ is nonzero at $p$ by our assumption.  

On the other hand, let $\bb = A(r,r]$ be the $r$-th column of $A$, removing its $r$-th entry.  Observe that the vector $\bb[C]$ has a unique nonzero entry at $p$ and $0$ elsewhere.  Thus, $\bb[C]$ and $\bx[C]$ have nonzero inner product, which means $\bb[C]\notin\Col(A[C] - \lambda I)$.  This further implies $\bb\notin\Col(A(r) - \lambda I)$ and $\mul_{A(r)}(\lambda) = \mul_A(\lambda) + 1$.  
\end{proof}

\begin{theorem}
\label{thm:typeiall}
Let $T$ be a tree and $\bx\in\mathbb{R}^{V(T)}$ a Type I Fiedler-like vector.  Then any matrix $A\in\mptn_L(T)$ with $\bx$ as its Fiedler vector with respect to the algebraic connectivity $1$ can be generated by Algorithm~\ref{alg:typeif}.  Moreover, the multiplicity of $1$ equals the number of $\mu_i$'s with $\mu_i = 1$ minus $1$.  
\end{theorem}
\begin{proof}
Let $A\in\mptn_L(T)$ be a matrix such that $\bx$ is its Fiedler vector with respect to $\lambda = 1$.  Let $r$ be the unique vertex in the characteristic set.  Let $B_1, \ldots, B_k$ be the branches of $T$ at $r$ and $V_i = V(B_i) - r$ for $i = 1, \ldots, k$.  

By Lemma~\ref{lem:charparter}, $\mul_{A(r)}(\lambda) = \mul_{A}(\lambda) + 1$.  By the Cauchy interlacing theorem, $\lambda = 1$ is the minimum eigenvalue of $A(r)$.  Since $A(r)$ is the direct sum of $A[V_i]$'s, each $A[V_i]$ has its minimum eigenvalue greater than or equal to $1$.  

Suppose $\bx[V_i] \neq \bzero$ for some $i$.  By direct computation, $\bx[V_i]$ is an eigenvector of $A[V_i]$ with respect to its minimum eigenvalue $\mu_i = 1$.  By Corollary~\ref{cor:dirunique}, the matrix generated by Theorem~\ref{thm:dirichletbi} is exactly $A[V_i]$.  

Suppose $\bx[V_i] = \bzero$ for some $i$.  Then we may assume $A[V_i]$ has its Perron value $\mu_i \geq 1$ and Perron vector $\bx_i$.  Thus, $A[V_i]$ can be constructed by Theorem~\ref{thm:dirichletbi} with $\mu_i$ and $\bx_i$.  Since $A[V_i]$ can be generated in either case, $A$ can be generated by Algorithm~\ref{alg:typeif}.

Finally, the multiplicity of $1$ is determined by $\mul_A(\lambda) = \mul_{A(r)}(\lambda) - 1$, which equals the number of $\mu_i$'s with $\mu_i = 1$ minus $1$.  
\end{proof}

\begin{corollary}
\label{cor:typeiunique}
Let $T$ be a tree with a vertex $r$ of degree $2$.  Every weighted Laplacian matrix $A\in\mptn_L(T)$ with the characteristic set $\{r\}$ uniquely determines the eigenpair $(\lambda,\vspan(\{\bx\}))$ of the algebraic connectivity and the corresponding one-dimensional eigenspace generated by a Type I Fiedler-like vector $\bx$.  Conversely, a pair $(\lambda, \vspan(\{\bx\}))$ of a value $\lambda > 0$ and a one-dimensional subspace generated by a Type I Fiedler-like vector $\bx$ uniquely determines the weighted Laplacian matrix $A\in\mptn_L(T)$ having $\lambda$ as the algebraic connectivity and $\vspan(\{\bx\})$ as the corresponding eigenspace.
\end{corollary}
\begin{proof}
If a matrix $A\in\mptn_L(T)$ has $\lambda_2(A) = 1$ and its characteristic set on $\{r\}$, then it can be generated by Algorithm~\ref{alg:typeif} with a Type I Fiedler-like vector, which is zero on $r$ and nonzero on the two branches.  By Theorem~\ref{thm:typeiall}, $\mul_A(\lambda) = 1$ since $\mu_1 = \mu_2 = 1$.  In general, any matrix $A\in\mptn_L(T)$ with $\lambda_2(A) = \lambda$ and its characteristic set on $\{r\}$ also has $\mul_A(\lambda) = 1$ by applying the arguments to $\frac{1}{\lambda}A$.  Conversely, let $\bx$ be a Type I Fiedler-like vector with the characteristic set $\{r\}$.  The output of Algorithm~\ref{alg:typeif} is unique with the input $\bx$.  By Theorem~\ref{thm:typeiall}, the matrix with a given algebraic connectivity $\lambda > 0$ and a Fiedler vector $\bx$ is unique.
\end{proof}

\begin{theorem}
\label{thm:typeiiall}
Let $T$ be a tree and $\bx\in\mathbb{R}^{V(T)}$ a Type II Fiedler-like vector.  Then any matrix $A\in\mptn_L(T)$ with $\bx$ as its Fiedler vector with respect to the algebraic connectivity $1$ can be generated by Algorithm~\ref{alg:typeiif}.  Moreover, the multiplicity of $1$ equals $1$.  
\end{theorem}
\begin{proof}
Let $A\in\mptn_L(T)$ be a matrix such that $\bx$ is its Fiedler vector with respect to $\lambda = 1$.  Let $V_-$ and $V_+$ be the set of vertices where $\bx$ has negative and positive entries, respectively.  Let $\{p,q\}$ with $p\in V_-$ and $q\in V_+$ be the characteristic set.  Thus, $A$ can be written as  
\[
    A = \begin{bmatrix}
        A[V_-] & -wE_{p,q} \\
        -wE_{p,q}\trans & A[V_+]
    \end{bmatrix},
\]
where $w$ is the weight assigned on $\{p,q\}$ and $E_{p,q}\in\mathbb{R}^{V_-\times V_+}$ has a unique $1$ at the $p,q$-entry and $0$ elsewhere.  Note that $A[V_-]$ is a weighted Laplacian matrix of $T[V_-]$ with $w$ added on its $p,p$-entry.  Similarly, $A[V_+]$ is a weighted Laplacian matrix of $T[V_+]$ with $w$ added on its $q,q$-entry.

Let $x_i$ be the $i$-th entry of $\bx$.  With $A\bx = \bx$, we have 
\[
    A[V_-]\bx[V_-] -wE_{p,q}\bx[V_+] = \bx[V_-].
\]
Since 
\[
    -wE_{p,q}\bx[V_+] = -wx_q\be_p = -w\frac{x_q}{x_p}E_p\bx[V_-],
\]
where $\be_p\in\mathbb{R}^{V_-}$ has a unique $1$ at the $p$-th entry and $0$ elsewhere, and $E_p\in\mathbb{R}^{V_-\times V_-}$ has a unique $1$ at the $p,p$-entry and $0$ elsewhere, we have  
\[
    (A[V_-] - w\frac{x_q}{x_p}E_p)\bx[V_-] = \bx[V_-].
\]
By Corollary~\ref{cor:dirunique}, the matrix generated by Theorem~\ref{thm:dirichletbi} and $-\bx[V_-]$ is $A[V_-] - w\frac{x_q}{x_p}E_p$, which indicates that the weight on the edge $\{p,q\}$ is $w_- = w - w\frac{x_q}{x_p}$.  Similarly, the matrix generated by Theorem~\ref{thm:dirichletbi} and $\bx[V_+]$ is $A[V_+] - w\frac{x_p}{x_q}E_q$, where $E_q\in\mathbb{R}^{V_+\times V_+}$ has a unique $1$ at the $q,q$-entry and $0$ elsewhere, and the weight on the edge $\{p,q\}$ is $w_+ = w - w\frac{x_p}{x_q}$.  Thus, it is straightforward to check that  
\[
    w = \frac{w_-w_+}{w_- + w_+}
\]
and Algorithm~\ref{alg:typeiif} indeed generates $A$.  

Finally, the fact that $1$ has multiplicity $1$ comes from the Cauchy interlacing theorem and the proof of Theorem~\ref{thm:typeiif} that $A(p)$ has $\lambda_1(A(p)) < 1 < \lambda_2(A(p))$.  
\end{proof}

\begin{corollary}
\label{cor:typeiiunique}
Let $T$ be a tree with an edge $\{p,q\}$.  Every weighted Laplacian matrix $A\in\mptn_L(T)$ with the characteristic set $\{p,q\}$ uniquely determines the eigenpair $(\lambda,\vspan(\{\bx\}))$ of the algebraic connectivity and the corresponding one-dimensional eigenspace generated by a Type II Fiedler-like vector $\bx$.  Conversely, a pair $(\lambda, \vspan(\{\bx\}))$ of a value $\lambda > 0$ and a one-dimensional subspace generated by a Type II Fiedler-like vector $\bx$ uniquely determines the weighted Laplacian matrix $A\in\mptn_L(T)$ having $\lambda$ as the algebraic connectivity and $\vspan(\{\bx\})$ as the corresponding eigenspace.
\end{corollary}
\begin{proof}
If a matrix $A\in\mptn_L(T)$ has $\lambda_2(A) = 1$ and its characteristic set on $\{p,q\}$, then it can be generated by Algorithm~\ref{alg:typeiif} with a Type II Fiedler-like vector.  By Theorem~\ref{thm:typeiiall}, $\mul_A(\lambda) = 1$.  In general, any matrix $A\in\mptn_L(T)$ with $\lambda_2(A) = \lambda$ and its characteristic set on $\{p,q\}$ also has $\mul_A(\lambda) = 1$ by applying the arguments to $\frac{1}{\lambda}A$.  Conversely, let $\bx$ be a Type II Fiedler-like vector with the characteristic set $\{p,q\}$.  The output of Algorithm~\ref{alg:typeiif} is unique with the input $\bx$.  By Theorem~\ref{thm:typeiiall}, the matrix with a given algebraic connectivity $\lambda > 0$ and a Fiedler vector $\bx$ is unique.
\end{proof}


\section{Characteristic contraction and subdivision}
\label{sec:connsub}

In this section, we consider the set of Type I weighted trees with its characteristic set on a vertex of degree $2$, denoted by $\wtot$, and the set of all Type II weighted trees, denoted by $\wtt$.  We show that there is a natural bijection between $\wtot$ and $\wtt$.  Note that when $(T,\bw)$ is a weighted tree in $\wtot$ or $\wtt$, its algebraic connectivity has multiplicity $1$ by Corollaries~\ref{cor:typeiunique} and \ref{cor:typeiiunique}, which was also observed in \cite{Fiedler90}.  In fact, the two corollaries provide a strong connection between the weighted Laplacian matrices and the Fiedler vectors.

Let $T$ be a tree.  Corollary~\ref{cor:typeiunique} gives a bijection between the weight assignments $\bw$ such that $(T,\bw)\in\wtot$ and the pairs $(\lambda, \vspan(\{\bx\}))$ such that $\lambda > 0$ and $\bx$ is a Type I Fiedler-like vector.  Similarly, Corollary~\ref{cor:typeiiunique} also gives a bijection between the weight assignments $\bw$ such that $(T,\bw)\in\wtt$ and the pairs $(\lambda, \vspan(\{\bx\}))$ such that $\lambda > 0$ and $\bx$ is a Type II Fiedler-like vector.  Figure~\ref{fig:otbijection} demonstrates some examples.  

Let $(T_1,\bw_1)$ be the weighted tree in Figure~\ref{fig:t1treechar}, removing vertices $8$ and $9$.  By Theorem~\ref{thm:typeif}, 
\[
    \bx_1 = (0, -2, -2, -1, 1.25, 1.875, 1.875)\trans
\]
is a Fiedler vector of $(T_1,\bw_1)$, so $(T_1,\bw_1)\in\wtot$.  On the other hand, let $(T_2,\bw_2)$ be the weighted tree in Figure~\ref{fig:t2treechar}.  Then we have seen that  
\[
    \bx_2 = (-2, -2, -1, 1.25, 1.875, 1.875)\trans
\]
is a Fiedler vector of $(T,\bw_2)$, so $(T_2,\bw_2)\in\wtt$.  Here we see that $\bx_1$ is obtained from $\bx_2$ by padding a $0$.  Naturally, there is a bijection between Type I Fiedler-like vectors with exact a $0$ entry at a vertex of degree $2$ and Type II Fiedler-like vectors.  With the help of Corollaries~\ref{cor:typeiiunique} and \ref{cor:typeiiunique}, we will see that this bijection also lead to a bijection between $\wtot$ and $\wtt$.  

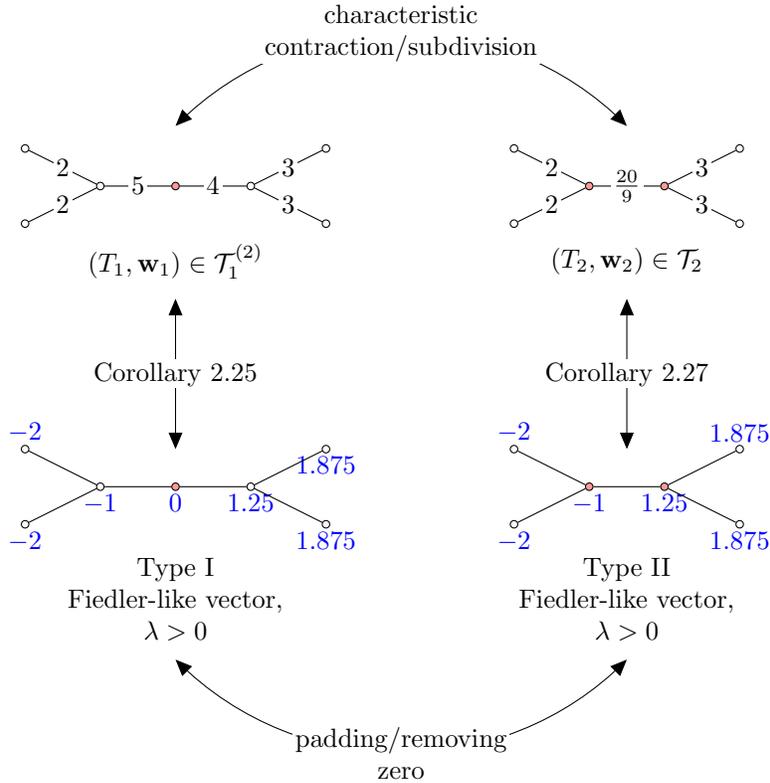
\begin{figure}[h]
\centering
\begin{tikzpicture}
\coordinate (a) at (-3,2);
\coordinate (b) at (-3,-2);
\coordinate (c) at (3,2);
\coordinate (d) at (3,-2);
\draw[triangle 45-triangle 45] ([yshift=-1.5cm]a) --node[edge weight]{Corollary~\ref{cor:typeiunique}} ([yshift=0.5cm]b);
\draw[triangle 45-triangle 45] ([yshift=-1.5cm]c) --node[edge weight]{Corollary~\ref{cor:typeiiunique}} ([yshift=0.5cm]d);
\draw[triangle 45-triangle 45, bend left=45] ([yshift=0.8cm]a) to node[edge weight, align=center]{characteristic\\contraction/subdivision} ([yshift=0.8cm]c);
\draw[triangle 45-triangle 45, bend right=45] ([yshift=-2.3cm]b) to node[edge weight, align=center]{padding/removing\\zero} ([yshift=-2.3cm]d);

\begin{scope}[shift={(a)}]
\node[selected] (a1) at (0, 0) {};
\node[] (a2) at (-1, 0) {};
\node[] (a3) at (-2, 0.5) {};
\node[] (a4) at (-2, -0.5) {};
\node[] (a5) at (1, 0) {};
\node[] (a6) at (2, 0.5) {};
\node[] (a7) at (2, -0.5) {};
\draw (a3) --node[edge weight]{$2$} (a2) --node[edge weight]{$2$} (a4);
\draw (a6) --node[edge weight]{$3$} (a5) --node[edge weight]{$3$} (a7);
\draw (a2) --node[edge weight]{$5$} (a1) --node[edge weight]{$4$} (a5);
\node[rectangle, draw=none] at (0,-1) {$(T_1,\bw_1)\in\wtot$};
\end{scope}

\begin{scope}[shift={(b)}]
\node[selected, label={below:$0$}] (b1) at (0, 0) {};
\node[label={below:$-1$}] (b2) at (-1, 0) {};
\node[label={above:$-2$}] (b3) at (-2, 0.5) {};
\node[label={below:$-2$}] (b4) at (-2, -0.5) {};
\node[label={below:$1.25$}] (b5) at (1, 0) {};
\node[label={below:$1.875$}] (b6) at (2, 0.5) {};
\node[label={below:$1.875$}] (b7) at (2, -0.5) {};
\draw (b3) -- (b2) -- (b4);
\draw (b6) -- (b5) -- (b7);
\draw (b2) -- (b1) -- (b5);
\node[rectangle, draw=none, align=center] at (0,-1.5) {Type I\\Fiedler-like vector,\\$\lambda > 0$};
\end{scope}

\begin{scope}[shift={(c)}]
\node[selected] (c1) at (-0.5, 0) {};
\node[] (c2) at (-1.5, 0.5) {};
\node[] (c3) at (-1.5, -0.5) {};
\node[selected] (c4) at (0.5, 0) {};
\node[] (c5) at (1.5, 0.5) {};
\node[] (c6) at (1.5, -0.5) {};
\draw (c2) --node[edge weight]{$2$} (c1) --node[edge weight]{$2$} (c3);
\draw (c5) --node[edge weight]{$3$} (c4) --node[edge weight]{$3$} (c6);
\draw (c1) --node[edge weight]{$\frac{20}{9}$} (c4);
\node[rectangle, draw=none] at (0,-1) {$(T_2,\bw_2)\in\wtt$};
\end{scope}

\begin{scope}[shift={(d)}]
\node[selected, label={below:$-1$}] (d1) at (-0.5, 0) {};
\node[label={above:$-2$}] (d2) at (-1.5, 0.5) {};
\node[label={below:$-2$}] (d3) at (-1.5, -0.5) {};
\node[selected, label={below:$1.25$}] (d4) at (0.5, 0) {};
\node[label={above:$1.875$}] (d5) at (1.5, 0.5) {};
\node[label={below:$1.875$}] (d6) at (1.5, -0.5) {};
\draw (d2) -- (d1) -- (d3);
\draw (d5) -- (d4) -- (d6);
\draw (d1) -- (d4);
\node[rectangle, draw=none, align=center] at (0,-1.5) {Type II\\Fiedler-like vector,\\$\lambda > 0$};
\end{scope}

\end{tikzpicture}
\caption{A diagram illustrating the bijection between $\wtot$ and $\wtt$.}
\label{fig:otbijection}
\end{figure}

Suppose $(T_1,\bw_1)\in\wtot$ is a weighted tree such that it has a Type I Fiedler vector whose characteristic set is on a vertex $r$ of degree $2$.  Let $p,q$ be the neighbors of $r$.  Let $w_p = \bw_1(p,r)$ and $w_q = \bw_1(q,r)$.  Define the tree $T_2$ as obtained from $T_1$ by replacing $\{p,r\}$ and $\{q,r\}$ with $\{p,q\}$.  Also, define the weight assignment $\bw_2$ of $T_2$ by 
\[
    \begin{aligned}
    \bw_2(p,q) &= \frac{w_pw_q}{w_p + w_q}, \\
    \bw_2(e) &= \bw_1(e) \text{ if } e \neq \{p,q\}.
    \end{aligned}
\]
We say $(T_2,\bw_2)$ is the \emph{characteristic contraction} of $(T_1,\bw_1)$.  One may verify that the characteristic contraction is the composition of computing the Fiedler vector $\bx_1$ and algebraic connectivity $\lambda > 0$ of $(T_1,\bw_1)$, removing the $0$ from $\bx_1$ to obtain $\bx_2$, and using $\bx_2$ and $\lambda > 0$ to reconstruct $(T_2,\bw_2)$.   

On the other hand, suppose $(T_2,\bw_2)\in\wtt$ is a weighted tree such that it has a Type II Fiedler vector $\bx$.  Let $\{p,q\}$ be the characteristic set, $w = \bw_2(p,q)$ and $x_i$ the $i$-th entry of $\bx$.  Define the tree $T_1$ as obtained from $T_2$ by subdividing $\{p,q\}$ into $\{p,r\}$ and $\{q,r\}$.  Also, define a weight assignment $\bw_1$ of $T_1$ by 
\[
    \begin{aligned}
    \bw_1(p,r) &= w \cdot (1 - \frac{x_q}{x_p}), \\
    \bw_1(q,r) &= w \cdot (1 - \frac{x_p}{x_q}), \text{ and }\\
    \bw_1(e) &= \bw_2(e) \text{ if } e \neq \{p,q\}.
    \end{aligned}
\]
By Corollary~\ref{cor:typeiiunique}, the algebraic connectivity in this case has multiplicity $1$.  Therefore, all Fiedler vector of $(T_2,\bw_2)$ is a multiple of $\bx$, so the construction of $(T_1,\bw_1)$ does not depend on the choice of $\bx$ and is well-defined.  With this definition, we say $(T_1,\bw_1)$ is the \emph{characteristic subdivision} of $(T_2,\bw_2)$.  Similarly, one may verify that the characteristic subdivision is the composition of computing the Fiedler vector $\bx_2$ and algebraic connectivity $\lambda > 0$ of $(T_2,\bw_2)$, padding a $0$ into $\bx_2$ to obtain $\bx_1$, and using $\bx_1$ and $\lambda > 0$ to reconstruct $(T_1,\bw_1)$.  

As we have seen that all mentioned steps are bijections, we have the following consequence.

\begin{theorem}
\label{thm:transform}
Characteristic contraction and characteristic subdivision preserve the algebraic connectivity and are the inverse of each other.  Let $(T_1,\bw_1)\in\wtot$ and $(T_2,\bw_2)$ its characteristic contraction, where $r$ is the vertex in $V(T_1)\setminus V(T_2)$.  Then the correspondence between the Fiedler vectors is as follows. 
\begin{itemize}
\item If $\bx_1$ is a Fiedler vector of $(T_1,\bw_1)$, then by removing the $r$-th entry, the resulting vector $\bx_2$ is a Fiedler vector of $(T_2,\bw_2)$.
\item If $\bx_2$ is a Fiedler vector of $(T_2,\bw_2)$, then by padding a new $0$ at the $r$-th entry, the resulting vector $\bx_1$ is a Fiedler vector of $(T_1,\bw_1)$. 
\end{itemize}   
\end{theorem}


\section{The inverse Fiedler vector problem of a cycle}
\label{sec:cycle}

In this section, we study the potential Fiedler vector of a cycle.  Interestingly, we found that for matrices $A\in\mptn_L(C_n)$, eigenvectors with respect to $\lambda_2(A)$ or $\lambda_3(A)$ share some common structure and are hard to be distinguished.

We usually label the vertices of $C_n$ by the elements $0, \ldots, n-1$ in the cyclic group $\mathbb{Z}_n$ of order $n$, following the cycle order.  A \emph{closed interval} $[i,j]$ stands for the vertices $\{i, i+1, \ldots, j\}$, while an \emph{open interval} $(i,j)$ stands for $\{i+1, i+2, \ldots, j-1\}$.  Thus, intervals like $(i,j]$ and $[i,j)$ are naturally defined.  

First note that this problem is simple for $C_3$ as it is a complete graph.  

\begin{observation}
Let $K_n$ be the complete graph with $n \geq 2$.  A vector $\bx$ is a Fiedler vector of $K_n$ if and only if $\bone\trans\bx = 0$ and $\bx \neq \bzero$.  
\end{observation}
\begin{proof}
A Fiedler vector of a complete graph must have $\bone\trans\bx = 0$ and $\bx \neq 0$.  Conversely, the classical Laplacian matrix $A$ of $K_n$ has spectrum $\{0,n^{(n-1)}\}$, so any vector $\bx$ with $\bone\trans\bx = 0$ and $\bx = 0$ is a Fiedler vector of $A$.
\end{proof}

To describe the possible Fiedler vector of a cycle, we will frequently use Lemma~\ref{lem:sum}, which is a well-known formula obtained from taking the sum of entries in $S$ from the both sides of $A\bx = \lambda\bx$; see, e.g., \cite[Statement~3.9]{Fiedler75class}.  

\begin{lemma}
{\rm\cite[Statement~3.9]{Fiedler75class}}
\label{lem:sum}
Let $(G,\bw)$ be a weighted graph, $A$ its weighted Laplacian matrix, and $\bx$ an eigenvector of $A$ with respect to $\lambda > 0$.  Then for any subset $S\in V(G)$, 
\[
    \sum_{\substack{
        \{i,j\}\in E(G) \\
        i\in S,\ j\notin S
    }} \bw(i,j)(x_i - x_j) =
    \lambda\sum_{i\in S}x_i,
\]
where $x_i$ is the $i$-th entry of $\bx$.
\end{lemma}

Next we show that the eigenvectors corresponding to $\lambda_2(A)$ or $\lambda_3(A)$ necessarily have nice structure for any $A\in\mptn_L(G)$.

\begin{definition}
\label{def:periodic}
Let $\bx = \begin{bmatrix} x_i \end{bmatrix} \in\mathbb{R}^{V(C_n)}$.  We say $\bx$ is \emph{periodic} if $\bone\trans\bx = 0$ and has the following properties.
\begin{enumerate}[label={\rm\arabic*.}]
\item Each of $C_n[V_+]$ and $C_n[V_-]$ is a path, while $C_n[V_0]$ is composed of isolated vertices with $0 \leq |V_0| \leq 2$.
\item Along the path $C_n[V_+]$, the values of $\bx$ are strictly increasing, followed by at most two maximum values, and then strictly decreasing.
\item Along the path $C_n[V_-]$, the values of $\bx$ are strictly decreasing, followed by at most two maximum values, and then strictly increasing.
\end{enumerate}
Here $V_+$, $V_-$, and $V_0$ are the set of indices where $\bx$ is positive, negative, and zero, respectively.
\end{definition}

\begin{example}
\label{ex:pnotb}
With $V(C_{12})$ labeled by $\mathbb{Z}_{12}$ following the cycle order, the vector  
\[
    (3,5,4,3,2,0,-3,-4,-4,-3,-2,-1)\trans
\]
is periodic.
\end{example}

\begin{lemma}
\label{lem:periodic}
Let $C_n$ be the cycle on $n$ vertices.  Then an eigenvector $\bx$ of $A$ with respect to $\lambda_2(A)$ or $\lambda_3(A)$ for some $A\in\mptn_L(C_n)$ is necessarily periodic.
\end{lemma}
\begin{proof}
Let $\bx = \begin{bmatrix} x_i \end{bmatrix}\in\mathbb{R}^{V(C_n)}$ be an eigenvector of $A$ with respect to $\lambda \in \{\lambda_2(A), \lambda_3(A)\}$ for some matrix $A\in\mptn_L(C_n)$.  We may assume that $A$ is the weighted Laplacian matrix of $(C_n,\bw)$.  Since $\bone$ is the eigenvector of $A$ with respect to $0$, we have $\bone\trans\bx = 0$ by the spectral theorem. 
Since any two consecutive vertices on $C_n$ form a zero forcing set, we know $C_n[V_0]$ is composed of isolated vertices by \cite{AIM}.  (Equivalently, if two consecutive entries of $\bx$ are zero, then $(A - \lambda I)\bx = \bzero$ implies $\bx = \bzero$.)  Moreover, if $x_i\in V_0$, then by applying Lemma~\ref{lem:sum} with $S = \{i\}$ we observe that $i$ has exactly one neighbor in $V_+$ and one neighbor in $V_-$.

Since $\bone\trans\bx = 0$, both $V_+$ and $V_-$ are not empty.  By \cite[Corollary~2.2]{Fiedler75class}, $C_n[V_+\cup V_0]$ has either one or two components since $\lambda\in\{\lambda_2(A), \lambda_3(A)\}$.  If it has two components, each component is surrounded by $V_-$, so each component contains some vertex in $V_+$.  (If a component only contains vertices in $V_0$, then there are two consecutive vertices in $V_0$ or a vertex in $V_0$ with two neighbors in $V_-$.  Either case is impossible.)  By the structure of 
a cycle, this implies that $C_n[V_-\cup V_0]$ has two components.  In total there are four components in $C_n[V_+\cup V_0]$ and $C_n[V_-\cup V_0]$ (weak nodal domains).  This is impossible by the Courant's nodal domain theorem \cite[Theorem~3.1]{BLS07} since $\bx$ is an eigenvector of $A$ with respect to $\lambda_2(A)$ or $\lambda_3(A)$.  Therefore, $C_n[V_+\cup V_0]$ has one component, which implies $C_n[V_-]$ has one component.  By a similar argument, each of $C_n[V_-\cup V_0]$ and $C_n[V_+]$ has one component.  Therefore, $C_n[V_+]$ and $C_n[V_-]$ are paths, and $0 \leq |V_0| \leq 2$.  

Since $C_n[V_+]$ is a path, we may assume the vertices on $V_+$ are $1, \ldots, k$, following the path order.  We note that a basin shape where $x_{i-1} \geq x_i$ and $x_i \leq x_{i+1}$ cannot happen for any $i$ with $1 < i < k$, because applying Lemma~\ref{lem:sum} with $S = \{i\}$ gives a contradiction  
\[
    \bw(i,i-1)(x_i - x_{i-1}) + \bw(i,i+1)(x_i - x_{i+1}) \leq 0 < \lambda x_i.
\]
Let $p$ be the smallest integer such that $x_p \geq x_{p+1}$.  Applying the above observation with $i = p+1$ gives $x_{p+1} > x_{p+2}$.  Inductively applying the observation with $i = p+1, \ldots, k-1$, we know $x_{p+1}, \ldots, x_k$ are strictly decreasing.  By our choice of $p$, $x_1,\ldots, x_p$ are strictly increasing.  Depending on $x_p = x_{p+1}$ or not, there may be two or one maximum values.  A similar argument also applies to $C_n[V_-]$.
\end{proof}

Lemma~\ref{lem:sum} guarantees further structure on a periodic vector. 

\begin{definition}
\label{def:balanced}
Let $\bx = \begin{bmatrix} x_i \end{bmatrix}\in\mathbb{R}^{V(C_n)}$ be a periodic vector on $C_n$.  A vertex $i$ is called a \emph{peak} if $x_i = \max(\bx)$ and a \emph{valley} if $x_i = \min(\bx)$.  If $\bx$ has a unique peak $p$, set $p' = p$; if $\bx$ has two peaks $p, p + 1$, set $p' = p + 1$.  If $\bx$ has a unique valley $q$, set $q' = q$; if $\bx$ has two valleys $q, q + 1$, set $q' = q + 1$.  We say $\bx$ is \emph{balanced} if 
\[
    \sum_{i\in (p,q]} x_i < 0 < \sum_{i\in [p',q')} x_i
\]
when $\bx$ has a unique peak or a unique valley, and 
\[
    \sum_{i\in (p,q]} x_i = 0 = \sum_{i\in [p',q')} x_i
\]
when $\bx$ has two peaks and two valleys.
\end{definition}

\begin{example}
The vector in Example~\ref{ex:pnotb} has $p = p' = 1$, $q = 7$, and $q' = 8$.  Since the sum of entries over $(p,q]$ is positive, the vector is not balanced.  

In contrast, the vector  
\[
(1,2,3,2,1,0,-2,-5,-2,0)\trans
\]
has $p = p' = 2$ and $q = q' = 7$ and
is a periodic and balanced vector on $C_{10}$.  
\end{example}

\begin{lemma}
\label{lem:balanced}
Let $C_n$ be the cycle on $n$ vertices.  Then a periodic eigenvector $\bx$ for some $A\in\mptn_L(G)$ with respect to a positive eigenvalue is necessarily balanced.
\end{lemma}
\begin{proof}
Let $\bx = \begin{bmatrix} x_i \end{bmatrix}\in\mathbb{R}^{V(C_n)}$ be a periodic vector satisfying $A\bx = \lambda \bx$ for some $A\in\mptn_L(C_n)$ and $\lambda > 0$. 
 Let $p,p',q,q'$ be as defined in Definition~\ref{def:balanced}.    

We first consider the case when $\bx$ has a unique peak or a unique valley.  We apply Lemma~\ref{lem:sum} with $S = (p,q]$ to obtain an equation.  Observe that
\[
    x_{p + 1} - x_{p} \leq 0 \text{ and }
    x_{q} - x_{q + 1} \leq 0 
\]
by the choice of $p$ and $q$, and one of them is nonzero by our assumption.  This means the left-hand side the the equation in Lemma~\ref{lem:sum} is negative.  Thus, we have   
\[
    \lambda\sum_{i\in (p,q]} x_i < 0.
\]
On the other hand, we apply Lemma~\ref{lem:sum} again with $S = [p',q')$ to get  
\[
    0 < \lambda\sum_{i\in [p',q')} x_i
\]
since 
\[
    x_{p'} - x_{p' - 1} \geq 0 \text{ and } x_{q' - 1} - x_{q'} \geq 0
\]
and they cannot be zero simultaneously.  This case is completed since $\lambda > 0$.  

For the other case when $\bx$ has two peaks and two valleys, the inequalities above becomes equalities because  
\[
    \begin{aligned}
    x_{p + 1} - x_{p} &= x_{q} - x_{q + 1} = 0 \text{ and } \\
    x_{p'} - x_{p' - 1} &=
    x_{q' - 1} - x_{q'} = 0.
    \end{aligned}
\]
Thus, the inequalities in Definition~\ref{def:balanced} also become equalities in this case.  
\end{proof}

\begin{theorem}
Let $C_n$ be the cycle on $n$ vertices.  Then $\bx$ is an eigenvector of $A\in\mptn_L(C_n)$ with respect to $\lambda_2(A)$ or $\lambda_3(A)$ if and only if $\bx$ is periodic and balanced.
\end{theorem}
\begin{proof}
By Lemmas~\ref{lem:periodic} and \ref{lem:balanced}, an eigenvector for some $A\in\mptn_L(C_n)$ with respect to $\lambda_2(A)$ or $\lambda_3(A)$ is necessarily periodic and balanced, so we will focus on the other direction.  Let $\bx = \begin{bmatrix} x_i \end{bmatrix} \in \mathbb{R}^{V(C_n)}$ be a periodic and balanced vector and $\lambda > 0$.  We will find a matrix $A\in\mptn_L(C_n)$ such that $A\bx = \lambda\bx$.  Label the vertices of $C_n$ with $\mathbb{Z}_n$ following the cycle order such that $0$ is a peak and $x_0 > x_1$.  We also label the edge $\{j,j+1\}$ as $e_j$ for $j\in\mathbb{Z}_n$.

Let $N\in\mathbb{R}^{V(C_n)\times E(C_n)}$ be the incidence matrix of $C_n$ under the orientation $(j+1,j)$ for all $e_j$.  That is, its $i,e_j$-entry is $-1$ if $i = j$, $1$ if $i = j + 1$, and $0$ otherwise.  Then any matrix $A\in\mptn_L(C_n)$ can be written as $A = NWN\trans$, where $W = \diag(\bw)$ is the weight diagonal matrix corresponding to the weight vector $\bw$.  By direct computation, the $e_j$-th entry of $N\trans\bx$ is $x_{j+1} - x_{j}$.  

On the other hand, let $N^+$ be the matrix in $\mathbb{R}^{E(C_n)\times V(C_n)}$ whose $e_j,i$-entry is $1$ for any $0 \leq j < i \leq n$.  By direct computation, we have $NN^+\bx = \bx$ for any $\bx$ satisfying $\bone\trans\bx = 0$.  Also, the right kernel of $N$ is spanned by $\bone$ since the only cycle on $C_n$ is itself.  Thus, for any given $\bx$ with $\bone\trans\bx = 0$, the solution set of $N\bz = \bx$ is $\{\bz = N^+\bx + h\bone: h\in\mathbb{R}\}$.  By direct computation, the entry of $N^+\bx + h\bone$ corresponding to the edge $e_j$ is $h + \sum_{i\in[j+1,n-1]}x_i$.  

Since $NWN\trans\bx = \lambda\bx$ implies $WN\trans\bx = \lambda(N^+\bz + h\bone)$ for some $h$.  Therefore, a weight vector $\bw$ exists if and only if $N\trans\bx$ and $N^+\bz + h\bone$ have the same signs ($+$, $-$, or $0$) at each corresponding entries for some $h$.  The weight vector $\bw$ can be obtained from $\lambda(N^+\bz + h\bone)\oslash(N\trans\bx)$ such that an arbitrary positive weight is assigned to each entry with $\frac{0}{0}$.  

Let $p,p',q,q'$ be as defined in Definition~\ref{def:balanced}.  (Based on our choice of $0$, we have $p' = 0$.)  We first examine the signs on $N\trans\bx$.  By the periodic structure of $\bx$, the $e_j$-th entry of $N\trans\bx$ is negative if $j\in [p',q)$, positive if $j\in [q',p)$, and zero if $e_j$ is $\{p,p'\}$ and $\{q,q'\}$ whenever these edges exist.  

Now we assign appropriate $h$ so that $N^+\bz + h\bone$ matches the signs of $N\trans\bx$.  If $\bx$ has two peaks, then let $h = 0$.  If $\bx$ has two valleys, then let $h = -\sum_{i\in(p,q]}x_i$.  Note that these two statement are consistent when $\bx$ has two peaks and two valleys since $\bx$ is balanced.  If $\bx$ has a unique peak and a unique valley, or equivalently, $p = p'$ and $q = q'$, then choose an arbitrary $h$ that satisfies  
\[
\begin{aligned}
    0 &< h < -\sum_{i\in(p,q]} x_i, \text{ and }\\
    x_p - \sum_{i\in [p',q')} x_i &< h < x_p.
\end{aligned}
\]

Note that these four inequalities are consistent and lead to some feasible $h$ since $\bx$ is balanced.  For example, 
\[
    x_p - \sum_{i\in [p',q')} x_i = - \sum_{i\in (p,q)} x_i = -\sum_{i\in(p,q]} x_i + x_q < -\sum_{i\in(p,q]} x_i,
\]
while the others are immediate from the definition of a balanced vector.  With the given $h$, one may check that $x_{j+1} - x_j$ and $h + \sum_{i\in[j+1,n-1]}x_i$ have the same sign for all $j$.

Thus, there exists a matrix $A\in\mptn_L(C_n)$ such that $A\bx = \lambda\bx$.  Now we show that $\lambda$ is either $\lambda_2(A)$ or $\lambda_3(A)$.  Since $\lambda_1(A) = 0$, it is enough to show $\lambda \leq \lambda_3(A)$.  Suppose $\bx$ has two zeros on $i$ and $j$.  Then $A(\{i,j\})$ is the direct sum of $A[V_+]$ and $A[V_-]$.  Also, $A[V_+]\bx[V_+] = \lambda\bx[V_+]$ and the fact that $\bx[V_+]$ is entrywisely positive implies that $\lambda = \lambda_1(A[V_+])$ by the Perron--Frobenius theorem; similarly, $\lambda = \lambda_1(A[V_-])$.  By the Cauchy interlacing theorem, $\lambda = \lambda_1(A(\{i,j\}) \leq \lambda_3(A)$.  

Suppose $\bx$ has a unique zero on $i$.  We first claim that $\lambda = \lambda_2(A(i))$.  Let $D = \diag(\bx(i))$ and $A' = D(A(i) - \lambda I)D$.  Thus, $A'\bone = D(A(i) - \lambda I)\bx(i) = \bzero$.  Note that $A'$ is a matrix in $\mptn(P_{n-1})$ with exactly one pair of positive off-diagonal entries.  By \cite[Theorem~2.2]{Fiedler75ac}, $A'$ has exactly one negative eigenvalue.  Since $A'$ and $A(i) - \lambda I$ are congruent, $A(i) - \lambda I$ has exactly one negative eigenvalue, and $\lambda = \lambda_2(A(i))$.  By the Cauchy interlacing theorem, $\lambda = \lambda_2(A(i)) \leq \lambda_3(A)$.   

Suppose $\bx$ contains no zero.  We let $D = \diag(\bx)$ and $A' = D(A - \lambda I)D$.  Thus, $A'\bone = D(A - \lambda I)\bx = \bzero$.  Observe that $A'\in\mptn(C_n)$ has exactly two pairs of positive off-diagonal entries.  Let the $(i,i+1)$-entry of $A'$ be a positive entry with value $c > 0$.  Define the symmetric matrix $B\in\mathbb{R}^{V(C_n)\times V(C_n)}$ such that 
\[
    B[\{i,i+1\}] = \begin{bmatrix}
        c & -c \\
        -c & c
    \end{bmatrix}
\]
while other entries are zero.  Then $A' + B$ is a matrix in $\mptn(P_n)$ with $(A' + B)\bone = \bzero$ and exactly one pair of positive off-diagonal entries.  By \cite[Theorem~2.2]{Fiedler75ac}, $0 = \lambda_2(A' + B)$.  Since $A' + B\in\mptn(P_n)$, its eigenvalues are all distinct \cite{Hochstadt74}, implying $\lambda_3(A' + B) > 0$.  Now by applying Weyl's inequality (see, e.g., \cite{Zhang11}) to $A' = (A'+B) + (-B)$, we have 
\[
    0 < \lambda_3(A' + B) + \lambda_2(-B) \leq \lambda_4(A').
\]
Since $A'$ and $A - \lambda I$ are congruent, this further implies that $0 < \lambda_4(A - \lambda I)$ and $\lambda \leq \lambda_3(A)$.  
\end{proof}


\section{Concluding remark}
\label{sec:conclude}
In this paper, we characterized all possible Fiedler vectors of a tree.  Given a Fiedler-like vector, we also provide the complete solution to all the weight assignments that match the vector.  As a side product, the relation between the Dirichlet matrix and the Perron vector is shown to be bijective.  We also characterized all possible eigenvectors corresponding to the second or the third smallest eigenvalues of a weighted Laplacian matrix of a cycle.

The results of trees and cycles show that the Fiedler vector (or eigenvector of $\lambda_3$) demonstrates very different behavior depending on the topological properties of the base graph, i.e., having a cycle or not.  For example, every tree has exactly one characteristic set of one vertex or two adjacent vertices, while every cycle has two such sets.  Moreover, $\lambda_2$ seems quite special for trees as the Fiedler vector can be fully described.  In contrast, $\lambda_2$ and $\lambda_3$ for cycles shares many common properties and are not so distinguishable by its eigenvectors.  To the extreme, we see that it is possible that $\lambda_2 = \cdots = \lambda_n$ for complete graphs. 
This shed some light on the definition of the Colin de Verdi\`ere parameter \cite{CdV, CdVF}, where the maximum nullity of the second smallest eigenvalue is used to characterize some the topological properties of a graph.  However, the properties of the eigenvectors of the weighted Laplacian matrices are far from being clear.  With a better understanding on these eigenvectors, it would provide a more concrete theoretical foundation for the applications such as graph partitioning, graph drawing, and spectral clustering.



\end{document}